\documentclass[a4paper,leqno]{amsart}

\usepackage[final=true]{hyperref,bookmark}
\usepackage[usenames,dvipsnames]{xcolor} 

\usepackage[T1]{fontenc}
\usepackage[utf8]{inputenc}

\usepackage[final]{microtype}

\usepackage{mathtools}
\mathtoolsset{mathic}

\usepackage{amsmath,
amsfonts,
amsthm,
amssymb,
extarrows,
amsthm,
comment,
mathscinet,
eucal,
amssymb
}
\usepackage{
bm,
upgreek,
tensor,
mathrsfs,
textcomp,
bbm,
braket
}

\usepackage[shortlabels]{enumitem}

\usepackage[bbgreekl]{mathbbol}

\makeatletter
\DeclareFontFamily{OMX}{MnSymbolE}{}
\DeclareSymbolFont{MnLargeSymbols}{OMX}{MnSymbolE}{m}{n}
\SetSymbolFont{MnLargeSymbols}{bold}{OMX}{MnSymbolE}{b}{n}
\DeclareFontShape{OMX}{MnSymbolE}{m}{n}{
    <-6>  MnSymbolE5
   <6-7>  MnSymbolE6
   <7-8>  MnSymbolE7
   <8-9>  MnSymbolE8
   <9-10> MnSymbolE9
  <10-12> MnSymbolE10
  <12->   MnSymbolE12
}{}
\DeclareFontShape{OMX}{MnSymbolE}{b}{n}{
    <-6>  MnSymbolE-Bold5
   <6-7>  MnSymbolE-Bold6
   <7-8>  MnSymbolE-Bold7
   <8-9>  MnSymbolE-Bold8
   <9-10> MnSymbolE-Bold9
  <10-12> MnSymbolE-Bold10
  <12->   MnSymbolE-Bold12
}{}

\let\llangle\@undefined
\let\rrangle\@undefined
\DeclareMathDelimiter{\llangle}{\mathopen}%
                     {MnLargeSymbols}{'164}{MnLargeSymbols}{'164}
\DeclareMathDelimiter{\rrangle}{\mathclose}%
                     {MnLargeSymbols}{'171}{MnLargeSymbols}{'171}
\makeatother

\usepackage{tikz}
\usetikzlibrary{patterns,matrix,through,arrows,decorations.pathreplacing,decorations.markings,shadows,shapes.geometric,positioning,calc,backgrounds,fit}

\tikzstyle directed=[postaction={decorate,decoration={markings,
    mark=at position #1 with {\arrow{>}}}}]
\tikzstyle rdirected=[postaction={decorate,decoration={markings,
    mark=at position #1 with {\arrow{<}}}}]

\tikzset{anchorbase/.style={baseline={([yshift=-0.5ex]current bounding box.center)}},
anchorzero/.style={baseline={([yshift=-0.5ex]0,0)}},
arrowinthemiddle/.style={postaction=decorate,decoration={markings,mark=at position 0.5 with {\arrow{>}}}},
arrowinthemiddlerev/.style={postaction=decorate,decoration={markings,mark=at position 0.5 with {\arrow{<}}}},
cross line/.style={preaction={draw=white,line width=4pt,-}},
int/.style={very thick},
zero/.style={thin,dotted},
uno/.style={thin},
smallnodes/.style={every node/.style={font=\footnotesize}},
}

 \usepackage[strict]{changepage}

\usepackage{graphicx}

\numberwithin{equation}{section}

\usepackage{etoolbox}
\makeatletter
\let\ams@starttoc\@starttoc
\makeatother
\usepackage{parskip}
\makeatletter
\let\@starttoc\ams@starttoc
\patchcmd{\@starttoc}{\makeatletter}{\makeatletter\parskip\z@}{}{}
\makeatother
\newtheoremstyle{myplain} {6pt plus 6pt minus 2pt}
{6pt plus 6pt minus 2pt}
{\itshape}
{}
{\bfseries}
{.}
{.5em}
{}

\theoremstyle{myplain}
\newtheorem{theorem}{Theorem}[section]
\newtheorem*{theorem*}{Theorem}
\newtheorem{lemma}[theorem]{Lemma}
\newtheorem{prop}[theorem]{Proposition}
\newtheorem{corollary}[theorem]{Corollary}
\newtheorem{conjecture}[theorem]{Conjecture}

\newtheoremstyle{mydefinition} {6pt plus 6pt minus 2pt}
{6pt plus 6pt minus 2pt}
{\itshape}
{}
{\bfseries}
{.}
{.5em}
{}

\theoremstyle{mydefinition}
\newtheorem{definition}[theorem]{Definition}

\newtheoremstyle{myexample} {6pt plus 6pt minus 2pt}
{6pt plus 6pt minus 2pt}
{}
{}
{\scshape}
{.}
{.5em}
{}

\theoremstyle{myexample}
\newtheorem{example}[theorem]{Example}

\newtheoremstyle{myremark} {6pt plus 6pt minus 2pt}
{6pt plus 6pt minus 2pt}
{}
{}
{\scshape}
{.}
{.5em}
{}

\theoremstyle{myremark}
\newtheorem{remark}[theorem]{Remark}

\begingroup
    \makeatletter
    \@for\theoremstyle:=mydefinition,myremark,myplain,myexample\do{%
        \expandafter\g@addto@macro\csname th@\theoremstyle\endcsname{%
            \addtolength\thm@preskip\parskip
            }%
        }
\endgroup

\newtheoremstyle{citing}
{3pt}{3pt}
{\itshape}
{0pt}{\bfseries}
{.}
{ }
{\thmnote{#3}}
\theoremstyle{citing}

\DeclareSymbolFontAlphabet{\mathbb}{AMSb}
\DeclareSymbolFontAlphabet{\mathbbol}{bbold}
\DeclareMathAlphabet{\mathpzc}{OT1}{pzc}{m}{it}

\DeclareSymbolFont{usualmathcal}{OMS}{cmsy}{m}{n}
\DeclareSymbolFontAlphabet{\mathucal}{usualmathcal}

\newcommand{\N}{\mathbb{N}}
\newcommand{\Z}{\mathbb{Z}}

\newcommand{\R}{\mathbb{R}}
\newcommand{\C}{\mathbb{C}}

\newcommand{\maps}{\colon}
\newcommand{\nadj}{\begin{tikzpicture}[anchorbase, scale=.8] \draw[line width=.33pt] (0,0) -- (.4,0); \draw[semithick] (.3,.15) -- (.1,-.15); \end{tikzpicture}}

\DeclareMathOperator{\Hom}{Hom}

\DeclareMathOperator{\lcm}{lcm}
\DeclareMathOperator{\HOM}{HOM}

\newcommand{\id}{\mathrm{id}}

\renewcommand{\epsilon}{\varepsilon}

\renewcommand{\phi}{\varphi}

\renewcommand{\emptyset}{\varnothing}

\newcommand{\calC}{{\mathcal{C}}}

\newcommand{\calK}{{\mathcal{K}}}
\newcommand{\calL}{{\mathcal{L}}}

\newcommand{\sfT}{{\mathsf{T}}}

\newcommand{\B}{\mathbb{B}}
\newcommand{\boldsfT}{{\boldsymbol{\sfT}}}
\newcommand{\calBe}{{\mathcal{B}e}}

\hypersetup{
  colorlinks = true,
  urlcolor = blue,
  pdfauthor = {Anthony M. Licata and Hoel Queffelec},
  pdfkeywords = {},
  pdftitle = {},
  pdfsubject = {},
  pdfpagemode = UseNone,
  bookmarksopen = true,
  bookmarksopenlevel = 3,
  pdfdisplaydoctitle = true }

\title[Braid groups of type ADE]{Braid groups of type ADE, Garside monoids, and the categorified root lattice}

\author{Anthony M. Licata}
\address{Mathematical Sciences Institute, Australian National University, John Dedman Building, 27 Union Lane, Canberra ACT 2601}
\email{amlicata@gmail.com}

\author{Hoel Queffelec}
\address{CNRS and Institut Montpelli\'erain Alexander Grothendieck, Universit\'e de Montpellier, Place Eug\`ene Bataillon, 34095 Montpellier, France}
\email{hoel.queffelec@umontpellier.fr}

\begin{document}

\begin{abstract}
We study Artin-Tits braid groups $\B_W$ of type ADE via the action of $\B_W$ on the homotopy category $\calK$ of graded projective zigzag modules (which categorifies the action of the Weyl group $W$ on the root lattice).  Following Brav-Thomas \cite{BT}, we define a metric on $\B_W$ induced by the canonical $t$-structure on $\calK$, and prove that this metric on $\B_W$ agrees with the word-length metric in the canonical generators of the standard positive monoid $\B_W^+$ of the braid group.  We also define, for each choice of a Coxeter element $c$ in $W$, a baric structure on $\calK$. We use these baric structures to define metrics on the braid group, and we identify these metrics with the word-length metrics in the Birman-Ko-Lee/Bessis dual generators of the associated dual positive monoid $\B_{W.c}^\vee$.  As consequences, we give new proofs that the standard and dual positive monoids inject into the group, give linear-algebraic solutions to the membership problem in the standard and dual positive monoids, and provide new proofs of the faithfulness of the action of $\B_W$ on $\calK$.  Finally, we use the compatibility of the baric and $t$-structures on $\calK$ to prove a conjecture of Digne and Gobet regarding the canonical word-length of the dual simple generators of ADE braid groups.
\end{abstract}

\maketitle

\setcounter{tocdepth}{1}
\tableofcontents

%
\section*{Introduction}
%

A basic tool in the combinatorial study of a Coxeter group $W$ is the linear action of $W$ on the root lattice.  The first important point about the action of $W$ on the root lattice is that it is faithful.  As a result, one can address basic combinatorial and group-theoretic questions about Coxeter groups, such as the word problem, by using tools of finite-dimensional linear algebra.

The Artin-Tits braid groups $\B_W$ associated to $W$ are much less well-understood than the Coxeter groups themselves; for example, in the generality of arbitrary Artin-Tits groups, the word problem is still open.  Perhaps one feature of the difficulty in studying $\B_W$ is the lack of a clear substitute for the root lattice: the action of $W$ on the root lattice can be $q$-deformed to the Burau representation of $\B_W$, but the Burau representation fails to be faithful outside of small rank.

In fact, there is a good candidate $\calK$ for a ``root lattice" for $\B_W$, though it is not a lattice but rather a triangulated category.  This triangulated category, which we refer to as the categorified root lattice, has a number of explicit realizations; for example, when $W$ is simply-laced, one can take for $\calK$ the derived category of modules over the associated preprojective algebra.  (For more general Coxeter groups, the categorified root lattice may be constructed as a quotient of the homotopy category of Soergel bimodules.)  At present, the question of whether or not the braid group $\B_W$ acts faithfully on the categorified root lattice is still open.  However, when the Coxeter group $W$ is finite, faithfulness is known: this was proven in type $A_2$ by Rouquier-Zimmermann \cite{RouquierZimmermann}, in type $A_n$ by Khovanov-Seidel \cite{KhS}, in type ADE by Brav-Thomas \cite{BT}, and for arbitrary finite $W$ by Jensen \cite{Jensen}. Faithfulness is also known in affine type A by work of Riche \cite{Riche}, Ishii-Ueda-Uehara~\cite{ishii2010stability} and Gadbled-Thiel-Wagner \cite{GTW}, as well as in the case of the free group by work of the first author \cite{Licata_free}.  Thus, at least in these cases, one can attempt to study $\B_W$ via its action on the triangulated category $\calK$ somewhat analogously to the way one studies $W$ via its action on the root lattice $\Lambda$.

The aim of the present paper is to take up such a study when $W$ is a finite Weyl group of type ADE. We take as our model for $\calK$ the homotopy category of graded modules over the zigzag algebra of the Dynkin diagram $\Gamma$.  Our main goal is to explain how the homological algebra of $\calK$ can be used to define metrics on $\B_W$, and to then combinatorially describe these metrics.  We consider two kinds of homological decompositions of $\calK$ into positive and negative pieces: t-structures and baric structures.  These decompositions turn out to be closely related to the standard and dual Garside monoids inside $\B_W$.

If a group $G$ acts on a triangulated category $\mathcal{T}$ by triangulated auto-equivalences, then one way to produce a pseudo-length function on $G$ is as follows: first, fix a t-structure $(\mathcal{T}^{\geq 0},\mathcal{T}^{\leq 0})$ with heart $\mathcal{T}^0$; given $g\in G$,
there is a smallest closed interval $[a,b]$, $a,b\in \mathbb{R}\cup \{\pm \infty\}$ such that
$g(\mathcal{T}^0)\subset \mathcal{T}^{[a,b]}$.  We then define the length $l(g)$ of $g$ to be the length of the interval $[a,b]$.  In good situations, the function $d(g,h) = l(h^{-1}g)$ will be a metric on $g$, though if the action of $G$ on $\mathcal{T}$ is not faithful, it will at best be a pseudo-metric.  Analogous definitions give rise to metrics on $G$ using other decompositions -- such as baric structures -- on $\mathcal{T}$, rather than t-structures.  For groups with interesting 2-representation theory (such as Artin-Tits groups and mapping class groups of surfaces), such metrics should carry interesting geometric information about the group.
 
In particular, in Theorem \ref{thm:2} we identify the metric on $\B_W$ induced from the canonical $t$-structure on $\calK$ with the canonical word-length metric on $\B_W$ coming from the positive lifts from the Weyl group (which are the canonical generators of the standard Garside monoid).  Though we give complete independent proofs of all the statements required for the proof of Theorem \ref{thm:2}, this theorem is inspired by -- and is in some sense largely a rederivation of -- the foundational work of Brav-Thomas on braid group actions on the derived category of a resolved Kleinian singularity \cite{BT}.

Our Theorem \ref{thm:1}, on the other hand, begins by choosing a Coxeter element $c\in W$, and using this choice to define a baric structure (rather than a $t$-structure) on $\calK$.  We use this baric structure to define a metric on $\B_W$; in Theorem \ref{thm:1} we identify this metric with the word-length metric on $\B_W$ coming from the generators of the dual Garside monoid on $\B_W$.  Taken together, Theorems \ref{thm:1} and \ref{thm:2} explain how both Garside structures on $\B_W$ can be studied in a parallel fashion via the action of $\B_W$ on the categorified root lattice.

Theorems \ref{thm:1} and \ref{thm:2} and the constructions that precede them have several important consequences.  For one, we obtain new proofs of the faithfulness of the action of $\B_W$ on $\calK$ (Corollary \ref{cor:faithfulness}).  As another consequence, we obtain new proofs of the injectivity of the canonical map from both the standard and dual Garside monoids $\B_W^+$, resp. ${\B_c^\vee}^+$ into the group $\B_W$ (Corollaries \ref{cor:moninj} and \ref{cor:dualmonoid}).  For the standard positive monoid, this injectivity was first established by Deligne \cite{Deligne72} and Brieskorn-Saito \cite{BrieskornSaito}, while for the dual positive monoid it is a theorem of Birman-Ko-Lee \cite{BKL} in type $A$ and by Bessis \cite{Bessis} and Brady-Watt \cite{BW} more generally.

We also give a solution to the membership problem in these monoids, by showing that 
\begin{itemize}
\item the canonical $t$-structure $(\calK^{\geq0},\calK^{\leq0})$ on $\calK$ has the property that 
\[
	\beta(\calK^{\geq0})\subset \calK^{\geq0} \iff \beta \in \B_W^+, \text{ and}
\]
\item the baric structure $(\calK_{\geq0},\calK_{\leq0})$ on $\calK$ associated to a Coxeter element $c\in W$ has the property that
\[
	\beta(\calK_{\geq0})\subset\calK_{\geq0} \iff \beta\in {\B_c^\vee}^+.
\]
\end{itemize}

To check each of the conditions $\beta(\calK^{\geq0})\subset \calK^{\geq0}$ and $\beta(\calK_{\geq0})\subset \calK_{\geq0}$ requires computing the action of $\beta$ on only finitely many objects of $\calK$; in turn, computing the action of $\beta$ on these finitely-many objects is equivalent to performing Gaussian elimination on a finite integer matrix (whose size depends on $\beta$).  Thus the above criteria provides a finite algorithm for determining whether or not a fixed word in the Artin braid generators represents a positive or dual-positive braid.  (Of course, computing the Garside or dual Garside normal form of a braid would also solve the above membership problem, so the point to emphasize here is the linear-algebraic nature of our solution.)  

An important motivation for our work is to study both Garside structures on $\B_W$ using the same categorical action, in an effort to shed some light on the relationship between these two structures.  In \cite{DigneGobet}, Digne-Gobet conjecture that any dual simple generator $\beta$ of the dual positive monoid can be written $\beta = x y^{-1}$, where $x$ and $y$ are both simple generators of the standard positive monoid.  This conjecture provides a basic relationship between the natural metrics on $\B_W$ coming from the standard and dual Garside monoids.  Digne-Gobet prove their conjecture in all irreducible type except for type $D_n$, with some of the proofs relying on a case-by-case computer analysis. In theorem \ref{thm:DGConj}, we use the compatibility of t- and baric structures on $\calK$ to give a uniform proof of their conjecture for all ADE braid groups, including the open type $D$.  A different proof for type D will also appear in a paper by Baumeister and Gobet~\cite{BG}.
 
The basic insight that triangulated categories with braid group actions should have $t$-structures compatible with the action of a distinguished positive monoid is originally due to Bezrukavnikov \cite{Bez_ICM}, who exploits this compatibility to prove a number of deep results in modular representation theory.  The baric structures of the current paper give an example of another kind of homological decomposition of a triangulated category compatible with the action of a positive monoid.

A more detailed outline of the contents of the paper is as follows:
\begin{itemize}
\item In Section \ref{sec:braids}, we collect some preliminary information about ADE Weyl groups and their braid groups, including the definitions of the standard and dual positive monoids and the associated word-length metrics.  We also recall ping-pong and dual ping-pong lemmas for braid groups (Lemmas \ref{lem:dualpingpong} and \ref{lem:pingpong}); these lemmas are used later in Section \ref{sec:pingpong} to establish the faithfulness of the action of $\B_W$ on the categorified root lattice $\calK$. 
\item In Section \ref{sec:zigzag} we recall the definition of the zigzag algebra $A_\Gamma$, its homotopy category $\calK$, and the action of the braid group on $\calK$.  We also define the $t$- and baric structures on $\calK$, whose compatibility with the braid group action is the central theme in the remaining sections.  We also define {\it root complexes} in $\calK$, which are important indecomposable complexes lying in the hearts of both the baric and $t$-structures.  The main new result in this section is Theorem \ref{thm:indecCplxes}, which establishes a bijection between root complexes and positive roots in the associated root system.  We also explain the relationship between this bijection and Gabriel's theorem, which concerns the representation theory of the undoubled quiver obtained by orienting the Dynkin diagram $\Gamma$ (see Proposition \ref{prop:equiv}).
\item Section \ref{sec:pingpong} is the technical heart of the paper.  We define explicit subsets of objects in $\calK$ and use these subsets to establish the assumptions of the ping-pong and dual ping-pong lemmas, giving new proofs of the faithfulness of the action of $\B_W$ on $\calK$.  The proofs of the injectivity of the maps from the standard and dual monoids into the group are also contained in this section.
\item The faithfulness proofs imply that the length functions defined by the $t$ and baric structures on $\calK$ induce metrics on $\B_W$.  In Section \ref{sec:metric}, we identify these metrics with word length metrics in standard and in dual generators (see Theorems \ref{thm:1} and \ref{thm:2}).
\item In Section \ref{sec:comparing}, we use the compatibility of $t$ and baric structures on $\calK$ to prove the conjecture of Digne-Gobet  \cite[Conjecture 8.7]{DigneGobet} for ADE braid group (see Theorem \ref{thm:DGConj}).
\end{itemize}

\noindent {\bf Acknowledgements:}		
We would like to thank Christian Blanchet,
Thomas Gobet,
Luis Paris,
Emmanuel Wagner
and Anne-Laure Thiel
for interesting discussions and comments.
This work was partially supported by the ARC DP 140103821 and the ANR Quantact.

%
\section{Artin-Tits braid groups in type ADE}\label{sec:braids}
%

%
\subsection{Weyl groups and braid groups}
%

In this section, we collect the definitions and classical facts about Weyl and braid groups, that we will freely use throughout the paper. For more details, we direct the reader to \cite{BjBr}, \cite{DDGKM}, and \cite{DehornoyParis}.

Let $\Gamma$ be a Dynkin diagram of type ADE, with vertex set $I$ of $n$ elements.
The Weyl group $W$ associated to $\Gamma$ has a Coxeter presentation:
\[
W=\langle s_i,\; i\in I\;|\; s_is_js_i=s_js_is_j,\; i\; \text{adjacent to}\; j \; \text{in}\; \Gamma,\; s_is_j=s_js_i,\; i\nadj j,\; s_i^2=1\rangle
\]

Denoting $S=\{s_i,i\in I\}$ the set of generators, $(W,S)$ is a Coxeter system. Recall that one can associate to it a root system as follows. Let $V^*$ be a real vector space of dimension $n$, with basis $\Pi=\{\alpha_i\}_{i\in I}$.  The basis vectors $\alpha_i$ are referred to as the \emph{simple roots}. We define a pairing $<-,->\maps V^*\times V^* \mapsto \R$ by
\[
<\alpha_i,\alpha_j>=
\begin{cases}
0 \; \text{if $i \nadj j$ in $\Gamma$;} \\
-1 \; \text{if $i \;\text{adjacent to}\; j$ in $\Gamma$;} \\
2 \; \text{if $i=j$.}
\end{cases}
\]

An action of $W$ on $V^*$ is defined by $s_i(x)=x-<\alpha_i,x>\alpha_i$, so that:

\begin{equation*}
s_i(\alpha_j) = 
\begin{cases}
-\alpha_i \;\text{if $i=j$};\\
\alpha_j+\alpha_i\;\text{if $i\;\text{adjacent to}\; j$ in $\Gamma$}; \\
\alpha_j \; \text{otherwise.}
\end{cases}
\end{equation*}

The root system $\Phi=\{w(\alpha_i)\}_{w\in W,i\in I}$ decomposes as $\Phi=\Phi^+\sqcup \Phi^-$, with $\Phi^+ = \mathbb{N}\Pi$ the positive roots and $\Phi^- = \Phi - \Phi^+ = -\Phi^+$ the negative roots. The root lattice is the free $\Z$-module generated by the $\alpha_i$. The above action on $V^*$ restricts to an action on the root lattice.

Let $l$ denote the word-length on $W$ in the Coxeter generators $s_i$.  The left and right descents sets of $w\in W$, denoted $D_R(w)$ and $D_L(w)$, are defined as follows
\[
	D_R(w) = \{s_i : l(ws_i)<l(w)\}, \ \ D_L(w) = \{s_i : l(s_iw)<l(w)\}.
\]

A \emph{reflection} in $W$ is any conjugate of a Coxeter generator $s_i$.  We denote by  $\sfT$ the set of reflections.  There is a one-to-one correspondence between $\Phi^+$ and $\sfT$ given by:
\[
w(\alpha_i)\in \Phi^+ \leftrightarrow ws_iw^{-1} \in \sfT
\]

The braid group $\B = \B_W$ associated to $\Gamma$ can be defined via the presentation
\[
\B=\langle \sigma_i,\; i\in I\;|\; \sigma_i\sigma_j\sigma_i=\sigma_j\sigma_i\sigma_j,\; i \;\text{adjacent to}\; j \; \text{in}\; \Gamma, \; \sigma_i\sigma_j=\sigma_j\sigma_i,\; i\nadj j\rangle
\]
There is a surjective group homomorphism $\pi\colon\B\mapsto W$ given by $\sigma_i \mapsto s_i$.

%
\subsection{Monoids and Garside structures in type ADE}
\label{subsec:Garside}
%

We recall the definition of a Garside structure here, and refer to \cite[I, 2.1]{DDGKM} for further details.

\begin{definition}
A Garside monoid is a pair $(M,\Delta)$, with $M$ a monoid, such that
\begin{itemize}
\item $M$ is left- and right-cancellative, that is, $\forall a,b,c\in M$, $ab=ac\Rightarrow b=c$ and $ba=bc\Rightarrow a=c$;
\item there exists $l:M\mapsto \N$ satisfying $l(fg)\geq l(f)+l(g)$ and $g\neq 1\Rightarrow l(g)\neq 0$;
\item any two elements of $M$ have a left and right $\lcm$ and a left and right $\gcd$;
\item $\Delta$ is a Garside element of $M$, meaning that the left and right divisors of $\Delta$ coincide and generate $M$;
\item the family of divisors of $\Delta$ in $M$ is finite. 
\end{itemize}
\end{definition}

A braid group of type ADE has two interesting Garside structures on it.  The underlying monoid $M$ of the first Garside structure is the classical braid monoid $\B_W^+$:
\[
\B_W^+=\langle \sigma_i,\; i\in I\;|\; \sigma_i\sigma_j\sigma_i=\sigma_j\sigma_i\sigma_j,\; i \;\text{adjacent to}\; j \; \text{in}\; \Gamma, \; \sigma_i\sigma_j=\sigma_j\sigma_i,\; i\nadj j\rangle_{monoid}
\]
It is a theorem of Deligne \cite{Deligne72} and Brieskorn-Saito \cite{BrieskornSaito} that the canonical map of monoids $\B_W^+ \longrightarrow \B_W$ is injective (and the categorical considerations later in the paper will give an alternative proof of that fact, see Corollary \ref{cor:moninj}).

The Garside element $\Delta$ of the first Garside structure is the positive lift from the longest element $w_0$ in the Weyl group, while the length function $l$ is the word length in the Artin generators $\{\sigma_i\}_{i\in I}$. An important fact about the positive monoid $\B_W^+$, which we will make use of in later sections, is that $\B_W^+$ is a lattice for left divisibility, so that any non-empty finite subset of $\B_W^+$ has both a greatest common divisor (gcd) and a least common multiple (lcm), both of which are themselves in the positive monoid $\B_W^+$.

In writing the Coxeter presentation of the Weyl group $W$ and its braid group $\B_W$, we have chosen a set of simple roots in order to give generators in our presentation.  More natural, from some points of view, is to choose for generators of $W$ the entire conjugacy class of reflections, and to present $\B_W$ by suitably lifting these reflections to the braid group.  This is the starting point of a very rich ``dual'' point of view on the braid group, originating in the work of Birman-Ko-Lee \cite{BKL} and Bessis \cite{Bessis} and further developed in work of Bessis-Digne-Michel \cite{BeDiMi} and Brady and Watt \cite{BW}.  Part of this dual point of view produces a second Garside structure on $\B_W$, whose definition we now recall.  (See also \cite[Section 3]{DigneGobet}.)

To begin, we choose an orientation $\Gamma_{\vec{o}}$ of $\Gamma$. The orientation $\Gamma_{\vec{o}}$ specifies a Coxeter element $c\in W$, defined by taking the product of all generators $s_i$ in such a way that if there is an oriented edge from $i$ to $j$ in $\Gamma_{\vec{o}}$, then $s_j$ is on the left of $s_i$ in the expression of $c$ (see Example \ref{ex:Cox}).

Once we have fixed a Coxeter element $c=s_{i_1}\cdots s_{i_n}$, we let $\gamma=\sigma_{i_1}\cdots \sigma_{i_n}$ 
be the positive lift of $c$ to the braid group.  We then lift the reflections $T\subset W$ to their ``dual positive lifts" $\boldsfT\subset \B_W$,
given by
(see \cite[Proposition 3.13]{DigneGobet}):
\begin{equation}\label{eq:dualpositivelift}
\boldsfT:=\{ \gamma^k\sigma_{i_1}\cdots \sigma_{i_{j}}\sigma_{i_{j+1}}\sigma_{i_{j}}^{-1}\cdots \sigma_{i_1}^{-1}\gamma^{-k}\},
\end{equation}
with $j\in \{1,2,\dots,n\}$. 

The positive monoid ${\B_c^\vee}^+ $ associated to the dual Garside structure on $\B_W$ is the monoid generated by the dual positive reflections $\boldsfT$.  In order to give a presentation of this monoid, and to discuss the length function associated to the dual Garside structure, we let$[1,c]_T\subset W$ denote the interval between $1$ and $c$ in the absolute partial order on $W$:
\[
	[1,c]_T = \{ u\in W \mid c = u v, \ \ l_{refl}(u) + l_{refl}(v) = l_{refl}(c) = n.\}
\]
Here $l_{refl} : W \rightarrow \mathbb{N}$ is the word length function in the generating set $T$ of all reflections. 

This generating set gives rise to an alternative presentation of the braid group: we have a generator $\tau_t$ for each reflection $t\in T\subset W$, and these generators are subject to those relations between distinct reflections in $W$ which are visible in the interval $[1,c]_T$ of the Weyl group.  More precisely, if we define
\[
	{\B_c^\vee}^+ = \langle \tau_t \mid \tau_{t_1} \tau_{t_2} = \tau_{t_2} \tau_{t_3} \text{ if } t_1t_2 = t_2t_3 \in [1,c]_T \rangle_{monoid}
\]
and 
\[
	{\B_c^\vee} =  \langle \tau_t \mid \tau_{t_1} \tau_{t_2} = \tau_{t_2} \tau_{t_3} \text{ if } t_1t_2 = t_2t_3 \in [1,c]_T \rangle_{group}
\]
then the map which takes $\tau_t$ to the dual positive lift (\ref{eq:dualpositivelift}) of the reflection $t$ in the braid group $\B_W$ defines a group isomorphism
\[
	\B_c^\vee \cong \B_W.
\]
Therefore, this second presentation gives rise to a morphism (of monoids)
\[
	{\B_c^\vee}^+ \longrightarrow \B_W
\]
It is a theorem of Birman-Ko-Lee in type $A$ \cite{BKL} and Bessis and Brady-Watt \cite{Bessis,BW} in type $ADE$ that this morphism is injective.  The categorical constructions of this paper will be another proof of this fact (see Corollary \ref{cor:dualmonoid}).
  
The second Garside structure on $\B_W$ then consists of the monoid ${\B_c^\vee}^+$, the length function given by the word-length $l_{refl}$ in the reflections $\{\tau_t\}$, and the Garside element given by $\gamma$.  The definition of the dual Garside structure depends on the choice of Coxeter element (or, equivalently, on the choice of orientation of $\Gamma$); however, since all Coxeter elements are conjugate, the resulting dual Garside structures are also all equivalent to each other.

A key subset of the dual positive monoid, which plays a large role in the rest of the paper, are the Bessis braids $\calBe^+\subset {\B_c^\vee}^+$ consisting of all dual positive divisors of $\gamma$:
\begin{equation}
\calBe^+=\{\beta : \gamma=\alpha\beta,\text{ with } \alpha,\beta\in {\B_c^\vee}^+ \text{ and } l_{refl}(\alpha)+l_{refl}(\beta)=l_{refl}(\gamma)\}.
\end{equation}
In fact the apparent left-right asymmetry in the above definition is illusory: for $\beta \in \calBe^+$, there exists $\alpha\in \calBe^+$ such that $\gamma = \alpha\beta$ and $ l_{refl}(\alpha)+l_{refl}(\beta)=l_{refl}(\gamma)$ if and only if there exists $\alpha'\in \calBe^+$ such that $\gamma =\beta \alpha'$ with $ l_{refl}(\beta)+l_{refl}(\alpha')=l_{refl}(\gamma)$.
 Note that the braid reflections $\boldsfT$, which generate the dual positive monoid, are themselves Bessis braids. 
 We denote by $\calBe^- = (\calBe^+)^{-1}$ the inverses of the Bessis braids. 

Bessis proves the following elementary property \cite[Fact 2.2.4]{Bessis}.
\begin{lemma}\label{lemma:8}
Let $\tau_{t_1}$ and $\tau_{t_2}$ be in $\boldsfT$. Then $\tau_{t_1}\tau_{t_2}\in \calBe^+ \Longrightarrow \tau_{t_2}^{-1}\tau_{t_1}\tau_{t_2}\in \boldsfT$.
\end{lemma}
Moreover, this property extends to Bessis braids. 
\[
\beta,\beta',\beta\beta'\in \calBe^+\Longrightarrow \beta'^{-1}\beta\beta'\in \calBe^+.
\]

\begin{definition}
Given $\beta\in \calBe^+$ a Bessis braid, we define the descent set of $\beta$, denoted $D_c^\vee(\beta)$, to consist 
of those reflections $\tau_t$ such that $\beta = \alpha \tau_t$, with $\alpha \in \calBe^+$.
\end{definition}

\begin{remark} \label{rem:descents}
The definition above is of {\emph right} descents, but, unlike the notion of descents associated with the standard positive monoid, Bessis braids have complete left-right symmetry in their descent sets.  That is, $D_c^\vee(\beta)$ may be equivalently defined as those reflections $\tau_t$ such that $\beta = \tau_t \alpha'$ with $\alpha'\in \calBe^+$.

For $\beta\in \calBe^+$, a reflection $\tau_t$ is a descent of $\beta$ if and only if $\tau_t$ appears in some minimal $l_{refl}$-length expression for $\beta$ in the generating set $\boldsfT$; moreover, in any minimal length expression for $\beta$, a given reflection $\tau_t$ will appear at most once.
\end{remark}

As with the original positive monoid $\B_W^+$, the dual positive monoid ${\B_c^\vee}^+$ is a combinatorial lattice. The following lemmas relate the lattice structure on ${\B_c^\vee}^+$ to the descent sets of a Bessis braid. 
\begin{lemma} \label{lemma:16}
Let $\beta\in \calBe^+$. Then $\lcm(D_c^\vee(\beta))=\beta$.
\end{lemma}
\begin{proof}
Since $\beta$ is a multiple of all its descents, it follows that $\lcm(D_c^\vee(\beta))$ divides $\beta$, and we may write $\beta=\beta'\lcm(D_c^\vee(\beta))$, with $\beta'\in \calBe^+$. Now, any right descent $\tau$ of $\beta'$ is also a left descent of $\beta'$, and hence $\tau$ is also a left descent of $\beta$.  But if $\tau$ is a descent of $\beta'$, then $\tau$ is also a descent of $\beta$, and hence a descent of $\lcm(D_c^\vee(\beta))$; but this implies that there must be a minimal length expression for $\beta$ with two appearances of $\tau$ - one in a minimal length expression for $\beta'$ and one in the expression for $\lcm(D_c^\vee(\beta))$.  Since this is not possible, we conclude that $\beta'$ has no descents, and thus that $\beta'=1$.
\end{proof}

\begin{lemma} \label{lemma:lcm_comp}
Let $\gamma=\beta'\beta$ with $\beta$, $\beta'\in \calBe^+$ and $l_{refl}(\gamma)=l_{refl}(\beta)+l_{refl}(\beta')$, and suppose that $\tau_t\in \calBe^+$ divides neither $\beta$ nor $\beta'$. Then there exists $\tau_{t'}\in \boldsfT$ dividing $\beta'$ such that $\tau_{t'}\tau_t\notin \calBe^+$.
\end{lemma}

\begin{proof}
 Let $\tau_t$ be a reflection which divides neither $\beta$ nor $\beta'$.  For another reflection $\tau_{t'}\in \boldsfT$, if  $\tau_{t'}\tau_t$ divides $\gamma$ for all $\tau_{t'}$ dividing $\beta'$, then  every such $\tau_{t'}$ divides $(\gamma \tau_{t}^{-1})$.  From this it follows that  $\beta'=\lcm(\{\tau_{t'} \text{ dividing }\beta'\})$ divides $(\gamma \tau_t^{-1})$, and we can write $(\gamma \tau_t^{-1})=\beta'\beta''$ with $l_{refl}(\gamma)-1=l_{refl}(\beta')+l_{refl}(\beta'')$.  But this implies that $\beta=\beta''\tau_t$, with 
 $\beta''\in \calBe^+$, which contradicts the assumption that $\tau_t$ does not divide $\beta$.
\end{proof}

\begin{example} \label{ex:Cox}
We give an example of the set $\calBe^+$ in type $A_3$, using the following orientation
\[
\begin{tikzpicture}
\node at (0,0) {$\bullet$};
\node at (1,0) {$\bullet$};
\node at (2,0) {$\bullet$};
\node at (0,.25) {$\scriptstyle 1$};
\node at (1,.25) {$\scriptstyle 2$};
\node at (2,.25) {$\scriptstyle 3$};
\draw [->] (.8,0) -- (.2,0);
\draw [->] (1.2,0) -- (1.8,0);
\end{tikzpicture}
\]
The associated Coxeter element is $c=s_1s_3s_2=s_3s_1s_2$, so that the dual Garside element $\gamma=\sigma_1\sigma_3\sigma_2$. The set $\boldsfT$ is then:
\[
\boldsfT=\{\sigma_1,\sigma_2,\sigma_3, \sigma_3\sigma_2\sigma_3^{-1},\sigma_{1}\sigma_2\sigma_1^{-1},\sigma_1\sigma_3\sigma_2\sigma_3^{-1}\sigma_1^{-1}\},
\]
while the lattice $\calBe^+$ is given as follows (the edges below denote the partial order on $\calBe^+$):
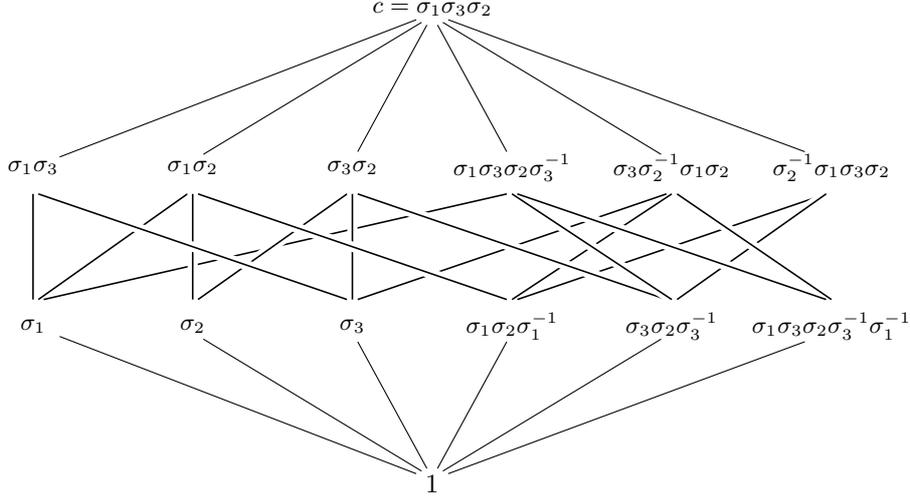
\[
\begin{tikzpicture}[anchorbase, scale=.7]
\node at (0,0) {$1$};
\node at (-7.5,3) {\small $\sigma_1$};
\node at (-4.5,3) {\small $\sigma_2$};
\node at (-1.5,3) {\small $\sigma_3$};
\node at (1.5,3) {\small $\sigma_1\sigma_2\sigma_1^{-1}$};
\node at (4.5,3) {\small $\sigma_3\sigma_2\sigma_3^{-1}$};
\node at (7.5,3) {\small $\sigma_1\sigma_3\sigma_2\sigma_3^{-1}\sigma_1^{-1}$};
\node at (-7.5,6) {\small $\sigma_1\sigma_3$};
\node at (-4.5,6) {\small $\sigma_1\sigma_2$};
\node at (-1.5,6) {\small $\sigma_3\sigma_2$};
\node at (1.5,6) {\small $\sigma_1\sigma_3\sigma_2\sigma_3^{-1}$};
\node at (4.5,6) {\small $\sigma_3\sigma_2^{-1}\sigma_1\sigma_2$};
\node at (7.5,6) {\small $\sigma_2^{-1}\sigma_1\sigma_3\sigma_2$};
\node at (0,9) {\small $c=\sigma_1\sigma_3\sigma_2$};
\draw (-.3,.2) -- (-7,2.8);
\draw (-.2,.25) -- (-4.3,2.75);
\draw (-.1,.3) -- (-1.4,2.7);
\draw (.3,.2) -- (7,2.7);
\draw (.2,.25) -- (4.3,2.75);
\draw (.1,.3) -- (1.4,2.7);
\draw [white,double,  double=black] (-7.5,3.5) -- (-7.5,5.5);
\draw [white,double, double=black] (-7.4,3.5) -- (-4.6,5.5);
\draw [white,double, double=black] (-7.3,3.5) -- (1.4,5.5);
\draw [white, double, ultra thick, double=black] (-4.5,3.5) -- (-4.5,5.5);
\draw [white, double, ultra thick, double=black] (-4.4,3.5) -- (-1.6,5.5);
\draw [white, double, ultra thick, double=black] (-1.6,3.5) -- (-7.4,5.5);
\draw [white, double, ultra thick, double=black] (-1.5,3.5) -- (-1.5,5.5);
\draw [white, double, ultra thick, double=black] (-1.4,3.5) -- (4.4,5.5);
\draw [white, double, ultra thick, double=black] (1.4,3.5) -- (-4.4,5.5);
\draw [white, double, double=black] (1.5,3.5) -- (4.5,5.5);
\draw [white, double, double=black] (1.6,3.5) -- (7.4,5.5);
\draw [white, double, ultra thick, double=black] (4.4,3.5) -- (-1.4,5.5);
\draw [white, double, ultra thick, double=black] (4.5,3.5) -- (1.5,5.5);
\draw [white, double, ultra thick, double=black] (4.6,3.5) -- (7.4,5.5);
\draw [white, double, ultra thick, double=black] (7.4,3.5) -- (1.6,5.5);
\draw [white, double, ultra thick, double=black] (7.5,3.5) -- (4.6,5.5);
\draw [white, double, double=black] (4.5,3.5) -- (1.5,5.5);
\draw [white, double, double=black] (7.4,3.5) -- (1.6,5.5);
\draw (-.3,8.8) -- (-7,6.2);
\draw (-.2,8.75) -- (-4.3,6.25);
\draw (-.1,8.7) -- (-1.4,6.3);
\draw (.3,8.8) -- (7,6.3);
\draw (.2,8.75) -- (4.3,6.25);
\draw (.1,8.7) -- (1.4,6.3);
\end{tikzpicture}
\]

\end{example}

%
\subsection{Standard and dual ping pong in type ADE}
%

Ping pong lemmas are ubiquitous in geometric group theory as tools to establish the faithfulness of (not-necessarily linear) actions of certain groups.   We will use two such lemmas -- one for each kind of Garside structure -- in establishing the faithfulness of the action of $\B_W$ on $\calK$; we record these lemmas here.  A ping pong proof of the faithfulness of a categorical action of the free group is appears in work of the first author, \cite{Licata_free}.

\begin{lemma} \label{lem:dualpingpong}
Let $Y$ be a set and suppose $\Psi\maps\B_W\mapsto \{\text{bijections}\;b_{ij}\maps Y\mapsto Y\}$ is a group homomorphism. Denote by $\Psi_\beta$ the image of an element $\beta\in \B_W$ under this homomorphism. Suppose that there exist disjoint non-empty subsets $\{X_w\}_{w\in \calBe^+}$ of $Y$ such that for all simple Bessis elements $u$ and $w$, $\Psi_u(X_w)\subset X_{lf(uw)}$, where $lf(uw)$ denotes the greatest left factor lying in $\calBe^+$. Then the group homomorphism $\Psi$ is injective.
\end{lemma}

\begin{proof}
See, for example, \cite[Section 3.2]{Licata_free}.
\end{proof}

\begin{remark} \label{rem:increfl}
Though we have not included the proof here, one can show that actually enough to check the inclusion condition for $u$ a reflection.
\end{remark}

The positive monoid for the standard Garside structure also has a ping pong lemma associated to it.  The proof is almost identical to that of Lemma \ref{lem:dualpingpong}.

\begin{lemma} \label{lem:pingpong}
Let $Y$ be a set and suppose $\Psi\maps\B_W\mapsto \{\text{bijections}\;b_{ij}\maps Y\mapsto Y\}$ is a group homomorphism. Denote by $\Psi_\beta$ the image of an element $\beta\in \B_W$ under this homomorphism. Suppose that there exist disjoint non-empty subsets $\{X_w\}_{w \in \mathcal{W}^+}$ of $Y$ such that for all elements $u,w\in \mathcal{W}^+$, $\Psi_u(X_w)\subset X_{lf(uw)}$, where $lf(uw)$ denotes the greatest left factor lying in $\mathcal{W}^+$. Then the group homomorphism $\Psi$ is injective.
\end{lemma}

%
\subsection{Standard and dual word-length metrics in type ADE} \label{subsec:prelimWordLength}
%

In this section we define the relevant word length metrics on $\B_W$. Two of the theorems in later sections will show how these metrics can be recovered from homological considerations: see Theorems \ref{thm:1} and \ref{thm:2}.

The two Garside structures considered in Section \ref{subsec:Garside} give rise to two word-length metrics on the braid group: the first is the word-length in the generating set consisting of positive Weyl braids $\mathcal{W}^+$ and their inverses, and the second is the word-length in the generating set of positive Bessis braids $\calBe^+$ and their inverses.  We briefly recall here some facts about geodesics in the braid group with respect to these two metrics.

Given a braid $\beta$, we denote by $L_{\mathcal{W}}(\beta)$ the minimal number of elements from $\mathcal{W}^+\cup (\mathcal{W}^+)^{-1}$ required to express $\beta$; $L_{\mathcal{W}}(\beta)$ is called the \emph{canonical length}, or sometimes the \emph{Charney length} of the braid $\beta$.

An important special fact about the braid groups associated to finite Coxeter groups (e.g. the braid groups of type ADE) is that any braid $\beta$ can be written as a product of a positive and a negative braid:
\[
	\beta=\beta^+\beta^{-}=\beta'^{-}\beta'^+ , \text{ with } \beta^+,\beta'^+,{\beta^-}^{-1},{\beta'^-}^{-1}\in \B_W^+. 
\]
Moreover, the expressions for $\beta$ above can be chosen with 
\[
	L_{\mathcal{W}}(\beta) =L_{\mathcal{W}}(\beta^+) + L_{\mathcal{W}}(\beta^-) = L_{\mathcal{W}}(\beta'^+)+L_{\mathcal{W}}(\beta'^-).
\]
Such expressions can be further refined into normal forms for $\beta$ by writing $\beta^+$ and $\beta^-$ in greedy form:
we say that an expression $\beta^+=\beta_1\cdots \beta_k$, with $\beta_i\in \mathcal{W}^+$, is \emph{left-greedy} if every left descent of $\beta_i$ is a right descent of $\beta_{i-1}$: 
\[
D_L(\pi(\beta_i))\subset D_R(\pi(\beta_{i-1})) \text{ for all }i = 1,\dots,k-1.
\]
Similarly, a \emph{right-greedy} expression for $\beta$ will be:
\[
\beta=\beta_1\cdots \beta_k,\; \text{ with } D_R(\pi(\beta_i))\subset D_L(\pi(\beta_{i+1})) \text{ for all } i=1,\dots,k-1.
\]

For a general braid $\beta\in \B_W$, we define a \emph{reduced minimal expression} for $\beta$ to be an expression of the following kind:
\[
\beta=\beta^-\beta^+,\; \beta^-\in \B_W^-,\;\beta^+\in \B_W^+
\]
with
\[
\beta^-=\beta_l^-\cdots \beta_1^-\;\text{right-greedy},\;\beta^+=\beta_1^+\cdots \beta_k^+\;\text{left-greedy},
\]
\[
	L_{\mathcal{W}}(\beta) = k+l, \text{ and }
\]
\[
D_R(\pi(\beta_1^{-})\cap D_L(\pi(\beta_1^+))=\emptyset.
\]
Every braid $\beta$ has a unique reduced minimal expression in the above sense. 

There is an entirely analogous story for geodesics in the word-length with respect to the Bessis braids $\calBe^+$ and their inverses in $\calBe^-=(\calBe^+)^{-1}$.  We define $L_{\calBe}(\beta)$ to be the minimal number of elements from $\calBe^+\cup \calBe^-$ needed to write $\beta$. $L_{\calBe}$ is the word-length metric in the generators from $\calBe^+\cup \calBe^-$.

Any braid $\beta$ can be written as a product of a dual positive and a dual negative braid:
\[
	\beta=\beta^+\beta^{-}=\beta'^{-}\beta'^+ , \text{ with } \beta^+,\beta'^+,{\beta^-}^{-1},{\beta'^-}^{-1}\in  {\B_c^\vee}^+. 
\]
Moreover, the expressions for $\beta$ above can be chosen with 
\[
	L_{\calBe}(\beta) =L_{\calBe}(\beta^+) + L_{\calBe}(\beta^-) = L_{\calBe}(\beta'^+)+L_{\calBe}(\beta'^-).
\]
Such expressions can be further refined into normal forms for $\beta$ by writing $\beta^+$ and $\beta^-$ in greedy form:
we say that an expression $\beta^+=\beta_1\cdots \beta_k$, with $\beta_i\in \calBe^+$, is \emph{left-greedy} if every descent of $\beta_i$ is a descent of $\beta_{i-1}$:
\[
D_c^\vee(\beta_i)\subset D_c^\vee(\beta_{i-1}) \text{ for all }i = 1,\dots,k-1.
\]
Similarly, a \emph{right-greedy} expression for $\beta$ will be:
\[
\beta=\beta_1\cdots \beta_k,\; \text{ with } D_c^\vee(\beta_i)\subset D_c^\vee(\beta_{i+1}) \text{ for all } i=1,\dots,k-1.
\]

For a general braid $\beta\in \B_W$, we define a \emph{dual reduced minimal expression} for $\beta$ an expression of the following kind:
\[
\beta=\beta^-\beta^+,\; \beta^-\in {\B_c^\vee}^-,\;\beta^+\in {\B_c^\vee}^+
\]
with
\[
\beta^-=\beta_l^-\cdots \beta_1^-\;\text{right-greedy},\;\beta^+=\beta_1^+\cdots \beta_k^+\;\text{left-greedy},
\]
\[
	L_{\calBe}(\beta) = k+l, \text{ and }
\]
\[
D_c^\vee(\beta_1^{-})\cap D_c^\vee(\beta_1^+)=\emptyset.
\]
Every braid $\beta$ has a unique reduced minimal expression in the above sense.

\subsection{The Digne-Gobet conjecture}

In relating the standard and dual Garside structures, a basic question is how to express simple dual elements, that is, elements from $\calBe^+$, in terms of generators of the classical monoid. Digne and Gobet investigated this question in \cite{DigneGobet}, and proposed the following conjecture.

\begin{conjecture}[\cite{DigneGobet}]
Let $\beta \in \calBe^+$, then $\beta \in [ \Delta^{-1},\Delta]$.
\end{conjecture}

In other words, dual simple braids can be expressed as the product of a simple braid and the inverse of a simple braid. In Section~\ref{sec:comparing}, we use the action of $\B$ on the categorified root lattice to give a uniform proof of this conjecture in type ADE.  We also note here that the  conjecture was already established via case-by-case analysis in types A and E (as well as in a number of other non-simply-laced types) in \cite{DigneGobet}.

%
\section{Zigzag algebras and categorical braid group actions}\label{sec:zigzag}
%

%
\subsection{Zigzag algebras}
%

Let $\bar{\Gamma}$ denote the oriented double of $\Gamma$, obtained by doubling all edges and orienting them in the two opposite directions. Let $Path(\bar{\Gamma})$ denote the path algebra of $\bar\Gamma$, which we regard as a $\C$-algebra. 
We draw basis elements of $Path(\bar{\Gamma})$ as paths, so that, for example, 
\begin{tikzpicture}[anchorbase,scale=.5]
\node at (0,0) {$\circ$};
\node at (0,.4) {$\scriptstyle a$};
\node at (1,1) {$\circ$};
\node at (1,.6) {$\scriptstyle b$};
\node at (2,0) {$\circ$};
\node at (2,.4) {$\scriptstyle c$};
\draw [->] (.1,.1) -- (.9,.9);
\draw [->] (1.1,.9) -- (1.9,.1);
\end{tikzpicture} 
denotes a length two path from vertex $a$ to vertex $c$ through vertex $b$.

 Then the zigzag algebra $A_{\Gamma}$ is the following quotient of $Path(\bar\Gamma)$:
\begin{equation}
A_{\Gamma}:=Path(\bar{\Gamma})/\left(
\begin{tikzpicture}[anchorbase,scale=.5]
\node at (0,0) {$\circ$};
\node at (0,.4) {$\scriptstyle a$};
\node at (1,1) {$\circ$};
\node at (1,.6) {$\scriptstyle b$};
\node at (2,0) {$\circ$};
\node at (2,.4) {$\scriptstyle c$};
\draw [->] (.1,.1) -- (.9,.9);
\draw [->] (1.1,.9) -- (1.9,.1);
\end{tikzpicture} = 0 \; \text{if $a\neq c$},\; 
\begin{tikzpicture}[anchorbase,scale=.5]
\node at (0,0) {$\circ$};
\node at (-.4,.3) {$\scriptstyle a$};
\node at (1,1) {$\circ$};
\node at (1.4,.7) {$\scriptstyle b$};
\draw [->] (0,.2) .. controls (.4,.8) and (.8,.95) .. (.85,.85) .. controls (.9,.75) and (.6,.2) .. (.2,0);
\end{tikzpicture} = \begin{tikzpicture}[anchorbase,scale=.5]
\node at (0,0) {$\circ$};
\node at (-.4,.3) {$\scriptstyle a$};
\node at (1,1) {$\circ$};
\node at (1.4,.7) {$\scriptstyle c$};
\draw [->] (0,.2) .. controls (.4,.8) and (.8,.95) .. (.85,.85) .. controls (.9,.75) and (.6,.2) .. (.2,0);
\end{tikzpicture} \; \forall a,b,c \in I
\right)
\end{equation}

So length two paths which start and end at distinct vertices are zero in $A_\Gamma$, and all length 2 paths which start and end at the same vertex $i$ are equal in $A_\Gamma$.  The above definition does not technically cover the type $A_1$ and $A_2$ cases: by convention, we take $A_\Gamma = \C[x]/x^2$ in type $A_1$; in type $A_2$ we define $A_\Gamma$ to be the quotient of $Path(\bar\Gamma)$ consisting of all length three paths.

\subsection{Gradings on $A_\Gamma$ and adjunctions}

There are a number of important non-negative gradings on the zigzag algebra.  One non-negative grading on $A_{\Gamma}$ is given by path length: all edges are assigned degree $1$; loops at a vertex are then of degree two. Shifts for this path-length grading will be denoted with angle brackets $\langle -\rangle$ in what follows.

Another non-negative grading on $A_\Gamma$ arises from the choice of an orientation $\vec{o}$ of $\Gamma$, resulting in the quiver denoted $\Gamma_{\vec{o}}$.  In this grading, we declare edges in $\bar{\Gamma}$ agreeing with the orientation $\Gamma_{\vec{o}}$ to be of degree 0, and edges of $\bar{\Gamma}$ not agreeing with the orientation $\Gamma_{\vec{o}}$ to have degree $1$.   In order to avoid confusion with the path-length grading on projective modules, we use the notation $\{-\}$ to denote a shift in orientation degree.  Though we will not always include it in the notation, it is important to remember that the orientation grading depends on a choice of orientation $\vec{o}$ of the Dynkin diagram, or, equivalently, on a choice of Coxeter element $c$ in the Weyl group $W$. By convention in type $A_1$ where $A_\Gamma=\C[x]/(x^2)$, $x$ has path length degree $2$ and orientation degree $1$.

After fixing an orientation $\vec{o}$, we endow $A_{\Gamma}$ with a bigrading by considering both the path-length and orientation gradings.  
We denote by $e_i$ the length 0 path at vertex $i$, which is an idempotent in $A_{\Gamma}$. The left $A_{\Gamma}$-module $P_i=A_{\Gamma}e_i$ is indecomposable and projective, as is the right $A_{\Gamma}$-module $e_iA_{\Gamma}$.  For bigraded $A_\Gamma$ modules $M,N$, we denote by $\Hom_{A_\Gamma}(M,N)$ the bidegree 0 
module homomorphisms from $M$ to $N$, and we define
\[
	\HOM_{A_\Gamma}(M,N) = \bigoplus_{k,l} \Hom_{A_\Gamma}(M,N\langle k \rangle \{l\})
\]
to be the space of all bigraded, but not necessarily degree 0, module homomorphisms.  For the indecomposable projective left modules $\{P_i\}$, we have:
\begin{equation} \label{eq:hompipjdual}
\HOM_{A_{\Gamma}-mod}(P_i,P_j)\cong
\begin{cases}
\C\oplus\C\{1\}\langle 2 \rangle \;\text{if $i=j$};\\
\C\langle 1 \rangle \; \text{if $i\mapsto j$ in $\Gamma_{\vec{o}}$};\\
\C\{1\} \langle 1 \rangle \; \text{if $j\mapsto i$ in $\Gamma_{\vec{o}}$};\\
0 \; \text{if $i\nadj j$}.
\end{cases}
\end{equation}

We set $Q_i = e_iA_{\Gamma} \langle -1 \rangle$, which is an indecomposable projective right $A_\Gamma$-module.
We have
\begin{equation} \label{eq:qipjdual}
Q_iP_j:= Q_i\otimes_{A_\Gamma} P_j\cong
\begin{cases}
\C\langle -1 \rangle \oplus \C\{1\}\langle 1 \rangle \;\text{if $i=j$};\\
\C \;\text{if $i\mapsto j$ in $\Gamma_{\vec{o}}$};\\
\C\{1\}\;\text{if $j\mapsto i$ in $\Gamma_{\vec{o}}$};\\
{0}\;\text{if $i\nadj j$}.
\end{cases}
\end{equation}
From these computations, the following lemma follows immediately.
\begin{lemma} \label{lem:dualPi_ex}
The functor given by tensoring on the left with $Q_i\langle 1 \rangle$ is right adjoint to tensoring with $P_i$, and tensoring with $Q_i\{-1\}\langle -1\rangle$ is left adjoint to it. In other words, for any left $\C$-module $X$ and any left $A_\Gamma$-module $Y$, there are isomorphisms of bigraded vector spaces
\begin{align*}
\HOM(P_iX,Y)&\simeq \HOM(X,Q_iY\langle1\rangle), \quad \text{and} \\
\HOM(Y,P_iX)&\simeq \HOM(Q_iY\{-1\}\langle -1\rangle,X) \\
\end{align*}
\end{lemma}

 Similarly, we have:
 \begin{lemma} \label{lem:dualotherside}
   The functor given by tensoring on the right with $Q_i\{-1\}\langle -1 \rangle$ is right adjoint to tensoring with $P_i$, and tensoring with $Q_i\langle 1\rangle$ is left adjoint to it. In other words, for any right $\C$-module $X$ and any right $A_\Gamma$-modules $Y$, there are isomorphisms of bigraded vector spaces
\begin{align*}
    \HOM(XQ_i\langle1\rangle,Y)&\simeq\HOM(X,YP_i) \quad \text{and} \\
    \HOM(Y,XQ_i\{-1\}\langle-1\rangle)&\simeq \HOM(YP_i,X).
\end{align*}
\end{lemma}

%
\subsection{Categories of complexes}
%

We denote by $\calK$ the homotopy category of bounded complexes of finitely-generated bigraded projective $A_{\Gamma}$-modules, where the bigrading on $A_{\Gamma}$ is given by the path-length grading and the orientation grading for a fixed orientation $\vec{o}$ of $\Gamma$.  Thus, an object of $\calK$ is a bounded complex of finitely generated bigraded projective modules
\[
	Y = (Y^m,\partial^m),\ \ \partial^m:Y^m \rightarrow Y^{m+1},\ \  \partial^{m+1}\circ \partial^m = 0.
\]
A morphism $f$ from $X$ to $Y$ is a collection of $A_{\Gamma}$ module maps $f^i:X^i \rightarrow Y^i$ intertwining the differentials.  Two maps $f,g:X\rightarrow Y$ are equal in $\calK$ if $f-g$ is nullhomotopic.
We let $[k]$ denote the auto-equivalence which shifts a complex $k$ degrees to the left: 
\[
	Y[k]^m = Y^{k+m}, \ \ \ \partial_{Y[k]} = (-1)^k\partial_Y. 
\]

The pair $(\calK,[1])$ is a finitely-generated, linear triangulated category.

Given a map of complexes $f : X\rightarrow Y$, the cone of $f$ is the complex $X[1] \oplus Y$ with the differential 
\[
	\partial (x,y) = (-\partial_{X}(x),f(x)+\partial_{Y}(y)).
\]

%
\subsection{Minimal complexes}
%

Let $Y$ be a complex of graded projective $A_{\Gamma}$-modules. We say that $Y$ is a {\it minimal complex} if $Y$ is indecomposable in the additive category $Com(A_\Gamma-mod)$ of bounded complexes of graded projective $A_{\Gamma}$-modules, that is:
\[
Y \cong Y_1\oplus Y_2 \in Com(A_\Gamma-Mod) \implies Y_1\cong 0\text{ or }Y_2\cong 0.
\]
The important point about minimal complexes is that if $Z\in \calK$ is indecomposable in the homotopy category $\calK$, then $Z$ has a representative $Y$ in the additive category $Com(A_\Gamma-mod)$ which is minimal.  In particular, when regarded as a complex of $A_{\Gamma}$-modules, $Y$ has no contractible summands.  Moreover, any two minimal representatives $Y_1$, $Y_2$ of such a $Z\in \calK$ are isomorphic in $Com(A_\Gamma)-mod$, so that the chain groups of a minimal representative are determined up to isomorphism as graded projective $A_{\Gamma}$ modules.  We thus refer to $Y$ as {\it the} minimal complex associated to its homotopy class, with the understanding that the chain groups of $Y$ are only determined up to non-canonical isomorphism.   As every indecomposable $Z\in \calK$ is homotopy equivalent to a minimal complex, in some arguments it will be convenient to study $Z$ by considering its associated minimal complex and studying the chain groups of that complex.

%
\subsection{The canonical $t$-structure on $\calK$}
\label{subsec:tbstruct}
%

The goal of this section is to briefly describe two ways to use a grading on $A_{\Gamma}$ to ``slice" objects of the homotopy category of projective modules into homogeneous pieces.  The first way, which uses well-known homological machinery, is via the {\it canonical $t$-structure} on $\calK$.  We assume some familiarity with the notion of a $t$-structure here, and refer the reader to \cite{GelfandManin}.  We also refer the reader to \cite{MOS} for a detailed discussion of hearts of $t$-structures on the homotopy category of complexes of projective modules over finite dimensional non-negatively-graded algebras.

Let $Y\in \calK$, and assume that $Y$ is minimal.  The chain groups of $Y$ are direct sums of shifts of graded projective modules $P_j \{l\} \langle k \rangle[m]$ for various $j\in I$ and $l,k,m\in \Z$.  A chain summand $P_j \{l\} \langle k \rangle[m]$ appearing in $Y$ is said to have \emph{level} $m-k$.  (Note that the orientation grading shift $\{l\}$ does not enter in the definition of the level.)
Given $N\in \Z$, we may define $\calK^N(Y)\in \calK$ to be the complex obtained by considering only those terms of the minimal complex of level $N$; the differential in the complex $\calK^N(Y)$ is the restriction of the differential of the minimal complex of $Y$ to those terms which have level $N$.  
We denote the full subcategory of $\calK$ consisting of complexes of the form $\calK^N(Y)$ by $\calK^N\subset \calK$.

For $N\in \Z$, the composition $\tau^N = \calK^N[-N]: \calK \rightarrow \calK$ is the $N$th {\it truncation functor} associated to the \emph{canonical $t$-structure} on $\calK$.  The heart $\calK^0$ of this $t$-structure consists of {\it linear complexes}, that is, complexes $Y\in \calK$ such that $Y\cong \tau^0(Y)$.  Note that for any complex $Y\in \calK$ and any $N\in \Z$, the complex $\tau^N(Y)$ is linear.  Given $Y\in \calK$, we will refer to the complexes $\tau^N(Y)$ as the {\it $t$-slices} of $Y$.  

At one point in the sequel we will also make use of the canonical t-structure on a homotopy category of $(A_\Gamma, A_\Gamma)$-bimodules.  The definition of this $t$-structure is almost identical to that given above for $A_\Gamma$ modules. In particular, the complexes of bimodules we will consider have chain groups isomorphic to $A_\Gamma$ and bimodules of the form $P_i\otimes_\C Q_i\{l\} \langle k \rangle[m]$; the level of such a summand is defined to be $m-k$, just as for modules.

\subsection{The $\vec{o}$-baric structure on $\calK$} \label{subsec:baricstructure}

In addition to $t$-structures, there is another way to slice a complex $Y\in \calK$ into homogeneous pieces, which is to ignore the homological degree completely and slice the minimal complex of $Y$ using only the internal orientation grading on $A_{\Gamma}$-modules.  The resulting decomposition of $Y$ arises not from a $t$-structure, but rather from a {\it baric structure}, a notion introduced by Achar and Treumann in \cite{AcharTreumann} for general triangulated categories.  

In our specific situation, the relevant baric structure is straightforward to describe:  for $k\in \Z$, let $\calK_k$ denote the full subcategory of $\calK$ consisting of complexes $Y = (Y^m,\partial^m)$ such that, in the minimal complex of $Y$, in all homological degrees all of the indecomposable projective summands in that homological degree have their orientation grading shifted exactly by $-k$.  In other words, in the minimal complex of $Y\in \calK_k$, the chain groups of the minimal complex are direct sums of projective modules of the form $P_i \langle n \rangle\{ k \}[m]$ for various $n,m\in \Z$.  
(There is no condition on the internal path-length grading shifts $\langle n \rangle$ or the homological shifts $[m]$ in the definition of the categories $\calK_k$.)  The subcategories $\{\calK_k\}$ define a 
baric structure in the sense of \cite{AcharTreumann}.  When we wish to emphasize the choice of the orientation $\vec{o}$ in the definition of the $\vec{o}$-baric structure, we will add it to the baric notation, and write, for example, $\calK_N^{\vec{o}}(Y)$.

Given $N\in \Z$ and $Y\in \calK$, we define $\calK_N(Y)$ to be the complex -- well-defined up to homotopy -- obtained by considering only those terms in the minimal complex of $Y$ whose orientation degree shift is $\{-N\}$, and restricting the boundary map to such terms.  We refer to the complexes 
$\tau_N(Y) = \calK_N(Y)\{N\}\in \calK_0$ as  the $o$-{\it baric slices} of the object $Y$.  To avoid confusion between $t$-slices and baric slices, we will denote the $N$th baric slices of $Y$ with $N$ as a subscript and the $N$th $t$-slice of $Y$ with $N$ as a superscript.

An important difference between $t-$ and baric- structures is in the relationship these structures have to the homological shift functor $[1]$; for example, the heart of a $t$-structure $\calK^0$ satisfies $\calK^0[1] = \calK^1$, while the baric heart $\calK_0$ satisfies $\calK_0[1] = \calK_0$.

Given an interval $[k,l]\subset \Z$, we set 
$\calK_{[k,l]}$ to be the full subcategory of $\calK$ consisting of those complexes $Y$ whose non-zero baric slices live in $\calK_m$ for $k\leq m \leq l$.  Similarly, $\calK^{[k,l]}$ is the full subcategory of $\calK$ whose non-zero $t$-slices live in $\calK^m$ for $k\leq m \leq l$.  The subcategories $\calK_{\leq k}$, $\calK_{\geq k}$, $\calK^{\leq k}$, $\calK^{\geq k}$ are defined similarly.

%
\subsection{The braid group action on $\calK$}
%

Following Khovanov-Seidel \cite{KhS},  the braid group $\B_W$ of $\Gamma$ acts by triangulated autoequivalences of $\calK$.
The braid generators $\sigma_i$ and $\sigma_i^{-1}$ act by tensoring with complexes of bimodules:
\begin{equation}\label{eq:braid}
\sigma_i \mapsto \left( A_{\Gamma}\rightarrow P_i\otimes Q_i\{-1\}\langle -1 \rangle \right), \\
\sigma_i^{-1} \mapsto \left( P_i\otimes Q_i\langle 1 \rangle \rightarrow A_{\Gamma}\right),
\end{equation}
where the bimodule $A_{\Gamma}$ is in homological degree $0$ in both complexes. 

The bimodule maps which define the differentials above are as follows:

\begin{align*}
1\mapsto&
\sum_{i'\;\text{adjacent to}\;i}
\begin{tikzpicture}[anchorbase,scale=.8]
\node at (0,0) {$\bullet$};
\node at (0,.3) {\small $i$};
\node at (-1,0) {$\bullet$};
\node at (-1,.3) {\small $i'$};
\draw [->] (-.9,.1) .. controls (-.6,.3) and (-.4,.3) .. (-.1,.1);
\node at (0,-.3) {};
\end{tikzpicture}
\otimes
\begin{tikzpicture}[anchorbase,scale=.8]
\node at (0,0) {$\bullet$};
\node at (0,.3) {\small $i$};
\node at (-1,0) {$\bullet$};
\node at (-1,.3) {\small $i'$};
\draw [<-] (-.9,.1) .. controls (-.6,.3) and (-.4,.3) .. (-.1,.1);
\node at (0,-.3) {};
\end{tikzpicture}
\;+\;
\begin{tikzpicture}[anchorbase,scale=.8]
\node at (0,0) {$\bullet$};
\node at (0,.3) {\small $i$};
\draw [->] (-.1,-.1) .. controls (-.6,-.3) and (-.9,-.1) .. (-.9,0) .. controls (-.9,.1) and (-.6,.3) .. (-.1,.1);
\node at (0,-.3) {};
\end{tikzpicture}
\otimes
\begin{tikzpicture}[anchorbase,scale=.8]
\node at (0,0) {$\bullet$};
\node at (0,.3) {\small $i$};
\node at (0,-.3) {};
\end{tikzpicture}
\\
&+\;
\begin{tikzpicture}[anchorbase,scale=.8]
\node at (0,0) {$\bullet$};
\node at (0,.3) {\small $i$};
\node at (0,-.3) {};
\end{tikzpicture}
\otimes
\begin{tikzpicture}[anchorbase,scale=.8]
\node at (0,0) {$\bullet$};
\node at (0,.3) {\small $i$};
\draw [->] (.1,-.1) .. controls (.6,-.3) and (.9,-.1) .. (.9,0) .. controls (.9,.1) and (.6,.3) .. (.1,.1);
\node at (0,-.3) {};
\end{tikzpicture}
\end{align*}
and
\begin{equation*}
\begin{tikzpicture}[anchorbase,scale=.8]
\node at (0,0) {$\bullet$};
\node at (0,.3) {\small $i$};
\node at (0,-.3) {};
\end{tikzpicture}
\otimes
\begin{tikzpicture}[anchorbase,scale=.8]
\node at (0,0) {$\bullet$};
\node at (0,.3) {\small $i$};
\node at (0,-.3) {};
\end{tikzpicture}
\mapsto
\begin{tikzpicture}[anchorbase,scale=.8]
\node at (0,0) {$\bullet$};
\node at (0,.3) {\small $i$};
\node at (0,-.3) {};
\end{tikzpicture}
\end{equation*}

\begin{prop}
The assignment (\ref{eq:braid}) defines a homomorphism from the braid group $\B_W$ to the group of (homotopy classes of) 
complexes of bigraded $(A_\Gamma, A_\Gamma)$ bimodules.
\end{prop}
The proof that these complexes of bimodules satisfy the braid relations is given in \cite[Proposition 2.4, Theorem 2.5]{KhS} (see also \cite{huerfano2001category}).
For future use, we note some particular computations of braids acting on the modules $P_i$ here.
\begin{align}
\sigma_i P_j=
\begin{cases}
P_i[-1]\{-1\}\langle-2\rangle \; \text{if $i=j$};\\
P_j\rightarrow P_i\{-1\}\langle -1 \rangle \; \text{with $P_j$ in homological degree zero, if $i\rightarrow j$}; \\
P_j\rightarrow P_i\langle -1 \rangle \; \text{with $P_j$ in homological degree zero, if $j\rightarrow i$}; \\
P_j \; \text{if $i\nadj j$}.
\end{cases}
\label{rel:KS1}
\\
\sigma_i^{-1} P_j=
\begin{cases}
P_i[1]\{1\}\langle 2\rangle\; \text{if $i=j$};\\
P_i\langle 1 \rangle \rightarrow P_j \; \text{with $P_j$ in homological degree zero, if $i\rightarrow j$}; \\
P_i\{1\}\langle 1 \rangle \rightarrow P_j \; \text{with $P_j$ in homological degree zero, if $j\rightarrow i$}; \\
P_j \; \text{if $i\nadj j$}.
\end{cases}
\label{rel:KS2}
\end{align}

Above we kept track of all gradings independently, and only indicated the shifts when they are non-zero.  
Note that $\sigma_i(P_j) \in \calK^0$ unless $j=i$, in which case $\sigma_i(P_i) \in \calK^1$.

This immediately implies the following lemma, that we will freely use throughout the paper.

\begin{lemma}
  If $\beta\in \B_W^+$, then $\beta (\calK^{\geq 0})\subset \calK^{\geq 0}$, and $\beta^{-1}9\calK^{\leq 0})\subset \calK^{\leq 0}$.
\end{lemma}

As mentioned in the introduction, one of the results of this paper is that the converse also holds: this follows, for example, from Proposition~\ref{prop:4}. 

\begin{remark}
As explained in \cite{KhS}, the action of $\B_W$ on the Grothendieck group is equivalent to the Burau representation of the braid group.  Since our action is on a category of bigraded chain complexes, passing to the Grothendieck group gives a 2-variable version of this Burau representation.
\end{remark}

%
\subsection{Root complexes in type ADE}
%

The goal of this section is to describe a relationship between the isomorphism classes of indecomposable linear complexes and the corresponding ADE root system.  To that end, fix an orientation $\vec{o}$ of the ADE Dynkin diagram $\Gamma$.  Recall that the choice of orientation gives rise to both a non-negative grading on the zigzag algebra $A_\Gamma$, and also to a baric structure on $\calK$, with baric heart denoted $\calK_0(\vec{o})$.  (For notional ease, we will sometimes drop the orientation $\vec{o}$ from the notation for the baric structure, and write $\calK_0$ instead of $\calK_0(\vec{o})$.)  Let $\calK^0_0= \calK^0 \cap \calK_0$ denote the intersection of the baric heart $\calK_0$ and the heart of the canonical $t$-structure $\calK^0$.  Objects of $\calK^0_0$ are complexes homotopic to a complex of indecomposable projectives whose underlying chain group is a direct sum of projective modules of the form $P_i\{0\}\langle -k\rangle$ lying in cohomological degree $k$.
Note that the shift $\langle 1 \rangle [1]$ preserves $\calK^0_0$.

In order to connect the representation theory of $A_\Gamma$ to the combinatorics of the associated root system, we will explain a relationship between the representation theory of the zigzag algebra $A_\Gamma$ (in particular, the category $\calK^0_0$) and the representations of the oriented quiver $\Gamma_{\vec{o}}$.  To do this, we regard the path algebra of $\Gamma_{\vec{o}}$ as a $\Z$-graded algebra, where the grading is by path-length.  A representation $(V_i,f_{ij})$ of $\Gamma_{\vec{o}}$ is said to be graded if the corresponding representation of the path algebra of $\Gamma_{\vec{o}}$ is graded.  Thus, in a graded representation of $\Gamma_{\vec{o}}$, the vector spaces $V_i$ at each vertex $i$ are themselves $\Z$-graded, and the linear maps assigned to edges of the quiver are of degree 1.   Let $\mbox{Rep}(\Gamma_{\vec{o}})$ denote the abelian category of graded representations of $\Gamma_{\vec{o}}$.  

We will define a functor
\[
	F_{\vec{o}} : \mbox{Rep}(\Gamma_{\vec{o}}) \longrightarrow \calK^0_0.
\]  
Let $(V_i, f_{ij})_{i,j}$ be a graded representation of $\Gamma_{\vec{o}}$.  If
\[
	V_i = \oplus_m V_i(m)
\]
where the $V_i(m)$ are the graded pieces, then the underlying projective module of the complex $F_{\vec{o}}( (V_i, f_{ij})_{i,j})$ is defined to be
\[
	\bigoplus_{i} \oplus_m P_i\langle -m \rangle [-m] \otimes_{\C} V_i(m),
\]
so that the underlying graded vector spaces $V_i$ of the representation become the multiplicity spaces of the projective $P_i$.  The differential $d$ on $F_{\vec{o}}( (V_i, f_{ij})_{i,j}) )$ 
is defined by setting its $i,j$ component to be $f_{ij} x_{ij}$ (where $x_{ij}$ is the edge from $i$ to $j$ in the zigzag algebra):
\[
	d = \oplus_{i,j} f_{ij} x_{ij}.
\]
It is clear from the definition that the resulting complex  $\big(F_{\vec{o}}( (V_i, f_{ij})_{i,j}) ,d \big)$ is an object of $\calK^0_0$.  
Moreover, we have the following.

\begin{prop}\label{prop:equiv}
The functor $F_{\vec{o}}:\mbox{Rep}(\Gamma_{\vec{o}}) \longrightarrow \calK^0_0$ is an equivalence of abelian categories.
\end{prop}
\begin{proof}
We construct the inverse $G_{\vec{o}}$ of $F_{\vec{o}}$ explicitly as follows.  Let $\calC\in \calK^0_0$ be a minimal complex.  The underlying chain group of $\calC$ is 
isomorphic to
\[
	\oplus_i \oplus_k P_i\langle -k \rangle [-k] \otimes N_i(k),
\]
where the multiplicity space $N_i(k)$ is a $\C$-vector space.
We define a graded representation $G_{\vec{o}}(\calC)$ by putting the graded vector space $\oplus_k N_i(k)$ at vertex $i$.  (Note that, since homotopic minimal complexes have isomorphic underlying projective modules, the isomorphism class of the above vector space is independent of the choice of minimal complex $\calC$.)
For neighboring vertices $i,j$ of $\Gamma_{\vec{o}}$, let $d_{ij}$ denote the component of differential $d$ of $\calC$ which maps from the $P_i$-isotypic component of $\calC$ to the $P_j$-isotypic component.  Since $\calC\in \calK^0$, it follows that $d_{ij}$ is of the form
\[
	d_{ij} = \sum_{k}f_{i,j}(k) x_{ij}, 
\]
where
\[
	 f_{ij}(k): N_i(k) \longrightarrow N_j(k+1).
\]
We complete the definition of $G_{\vec{o}}$ by putting the linear map $f_{ij} = \sum_k f_{ij}(k)$ on the edge from $i$ to $j$ in $\Gamma_{\vec{o}}$.  
It is then clear from the construction that $G_{\vec{o}}$ and $F_{\vec{o}}$ are inverse functors.
\end{proof}

Let $\tilde{\calL}$ denote the set of isomorphism classes of indecomposable objects of $\calK^0_0$, and let $\calL = \tilde{\calL}/\langle 1 \rangle [1]$ denote the set of equivalence classes of objects in $\tilde{\calL}$ up to the shift $\langle 1 \rangle [1]$.  The assignment $\chi(P_i) = \alpha_i$ extends to a well-defined map
\[
	\mbox{dim}: \calL \longrightarrow \mbox{span}_\mathbb{N}\{\alpha_i\}
\]
from $\calL$ to the non-negative cone in the root lattice of the corresponding ADE root system.  (Here $\alpha_i$ are the simple roots.)  We denote by 
\[
	\Phi^+ \subset \mbox{span}_\mathbb{N}\{\alpha_i\}
\]
the set of positive roots.  As a corollary of Proposition \ref{prop:equiv}, together with Gabriel's theorem~\cite{Gabriel}, we obtain the following.

\begin{theorem} \label{thm:indecCplxes}
For $\Gamma$ a Dynkin diagram of type $ADE$, the map $\chi$ defines a bijection between $\calL$ and $\Phi^+$.
\end{theorem}

\begin{remark}\label{rem:reflectionfunctors}
Though we will not use them in the sequel, we mention here how the Bernstein-Gelfand-Ponomarev reflection functors \cite{BGP} between the categories $\mbox{Rep}(\Gamma_{\vec{o}})$ for the various orientations $\vec{o}$ appear in relation to the braid group action on $\calK$.  
For an orientation $\vec{o}$, let 
\[
	\tau_{\vec{o}}: \calK \longrightarrow \calK^0_0(\vec{o})
\]
denote the truncation (``0th cohomology") functor for the $t$-structure whose heart is $\calK^0_0(\vec{o})$.

For $i$ a sink of the orientation $\vec{o}$, we define
\[
	S^+_i = \tau_{s_i\vec{o}} \sigma_i : \calK^0_0(\vec{o}) \longrightarrow \calK^0_0(s_i \vec{o}).
\]
Note that the target  of $S_i$ is the intersection of the $t$-heart $\calK^0$ and the baric heart $\calK_0(s_i \vec{o})$ of the quiver with orientation $s_i \vec{o}$.
Similarly, for $i$ a source of $\vec{o}$, we define 
\[
	S_i^- = \tau_{s_i\vec{o}} \sigma_i^{-1}: \calK^0_0(\vec{o}) \longrightarrow \calK^0_0(s_i \vec{o}).
\]
$S_i^{\pm}$ are additive functors, but they are not equivalences.  Under the equivalences $F_{\vec{o}}$ between the hearts $\calK^0_0(c)$ and the categories 
$\mbox{Rep}(\Gamma_{\vec{o}})$ defined above, the functors $S_i^{\pm}$ are sent to the BGP reflection functors.
\end{remark}

For future use, we also record here the basic compatibility between the braid Coxeter element $\gamma$ defined by an orientation $\vec{o}$ of the Dynkin diagram and the $\vec{o}$-baric structure determined by the same orientation: the action of $\gamma$ on $\calK$ moves complexes up in $\vec{o}$-baric level up by one.

\begin{lemma} \label{lemma:cShift}
Let $\calC\in \calK_0$. Then $\gamma\cdot \calC \in \calK_{1}$.
\end{lemma}
\begin{proof}
It suffices to show that $\gamma \cdot P_i\in \calK_1$ for all $i$.  So consider $\gamma \cdot P_i$. We have $\gamma = \sigma_1 \dots \sigma_{i-1} \sigma_i \sigma_{i+1}\dots \sigma_n$.  Since for $j>i$ the only morphisms from $P_i$ to $P_j\langle k \rangle \{l \}$ 
have $l =1$, it follows that 
$\sigma_{i+1}\dots \sigma_n P_i$ has a minimal complex of the form
 $P_i\rightarrow X$, where $X\in \calK_1$ is an object whose minimal complex has as underlying chain group a direct sum of modules of the form $P_j\{-1\}$, with $j>i$.  
 Since $\sigma_i(P_i) \cong P_i\{-1\} \in \calK_1$, applying $\sigma_i$ to $\sigma_{i+1}\dots \sigma_n P_i$ results in a complex which lives in $\calK_1$, and whose underlying chain groups are direct sums of $P_j\{-1\}$ with $j\geq i$. 
 But now when we apply 
 $\sigma_1\dots \sigma_{i-1}$ to this, the result will remain in $\calK_1$, since all homs from $P_k$ to $P_j$ for $k<j$ are of orientation degree 0.
 Thus $\gamma \cdot P_i\in \calK_1$, as desired.
 \end{proof}

\subsection{The reflection complexes $\calC_t$ and $\mathfrak{C}_t$}
Let $\phi$ be a positive root, and $t\in W_\Gamma$ the associated reflection in the Weyl group.  As explained in Section \ref{subsec:Garside}, the choice of orientation $\vec{o}$ (equivalently, the choice of Coxeter element in $W_\Gamma$) defines a lift of $t$ to the braid group, via 
\[
	t \mapsto \tau=\gamma^k\sigma_{i_1}\dots\sigma_{i_j}\sigma_{i_{j+1}}\sigma_{i_j}^{-1}\dots \sigma_{i_1}^{-1}\gamma^{-k}.
\]
The basic compatibility between this lift and Theorem \ref{thm:indecCplxes} is that the indecomposable complex associated to $\phi$ in Theorem \ref{thm:indecCplxes} is
\[
\calC_{t}=\gamma^k\sigma_{i_1}\dots \sigma_{i_j}P_{i_{j+1}}\{-k\}.
\]
Note, however, that under Theorem \ref{thm:indecCplxes}, the complex assigned to a positive root is only well-defined up to a shift.  Thus it is perhaps more natural to associate to the positive root $\phi$ the direct sum $\mathfrak{C}_t$ of all shifts of the complex $\calC_t$ above:
  \[
\mathfrak{C}_t:=\oplus_{l\in \Z}\calC_t[l]\langle l\rangle.
\]
The complex $\mathfrak{C}_t$ is a direct sum of bounded complexes, but is itself not bounded.  Thus technically $\mathfrak{C}_t$ is not an object of $\calK$.  However, note that for fixed $Y\in \calK$, the morphism spaces between $Y\in \calK$ and $\mathfrak{C}_t$ is a well-defined, finite-dimensional bigraded vector space.

In the definition of $\mathfrak{C}_t$ we have included shifts $\langle l\rangle$ in the path length grading so that the summands lie in the heart intersection $\calK^0_0 = \calK^0\cap \calK_0$. In latter parts of the paper which are not concerned with the canonical $t$-structure -- for example in the proof of Proposition \ref{prop:toolsdualpp2}, we will omit these path-length grading shifts.

%
\section{Categorical ping pong}\label{sec:pingpong}
%

%
\subsection{Standard ping pong and the canonical t-structure}
%

The orientation shifts $\{k\}$ have no bearing on the statements in this subsection, so we choose to omit them here.
For $w\in W$, let $w_+$ denote the positive lift to the braid group.  For the purposes of this section, we define a \emph{negative braid complex}
to be a complex of the form $\beta P_i$ for some negative braid $\beta\in \B_W^-$.  Negative braid complexes live in the non-negative part $\mathcal{K}^{\leq 0}$ of the canonical $t$-structure on $\calK$.

We denote by $rf(\beta)$ and $lf(\beta)$ the right, respectively, left factor in the greedy normal form of a braid.  If $\beta$ is a positive braid, the factor $lf(\beta)$ is the longest positive lift from $W$ which divides $\beta$ on the left in the positive braid monoid; moreover, $lf(\beta)$ is the least common multiple of all positive lifts from $W$ which divide $\beta$ on the left in the positive monoid.  Similarly, if $\beta$ is a negative braid, the factor $rf(\beta)$ is the longest negative lift from $W$ which divides $\beta$ on the right; in particular, for $\beta$ positive, 
$lf(\beta) = rf(\beta^{-1})^{-1}$.

For $w\in W$, we define sets 
$X^w$ as follows: a complex $\mathcal{C} \in \mathcal{K}^{\geq 0}$ is in $X^w$ if for all \emph{negative braid complexes} $Y$,
\[
	\Hom(\mathcal{C}, Y) = 0 \iff {w_+}^{-1} (Y) \in \mathcal{K}^{<0}.
\]

So, for example, the complex $\mathcal{C} = \oplus_{i\in I} P_i$ is in the set $X^{\mbox{id}}$, since 
\[
	\Hom(\oplus_{i\in I} P_i, Y) = 0 \iff Y\in \mathcal{K}^{<0} \iff \mbox{id} (Y) \in \mathcal{K}^{<0}.
\]

We have the following. 
\begin{prop} \label{prop:Bruhatping}
Let $\mathcal{C}\in X^u$.  Then, for $w\in W$, $w_+(\mathcal{C}) \in X^{lf(w_+u_+)}$.
\end{prop}

Note that Proposition \ref{prop:Bruhatping} has as a consequence the fact that the sets $X^w$ are non-empty, as the set $X^\id$ is itself non-empty.  In order to prove Proposition \ref{prop:Bruhatping}, we need a couple of preliminary results. 

\begin{lemma}\label{lem:neg}
Let $\beta\in \B_W^-$ be an element of the classical negative braid monoid.  Then 
\[
	\beta P_i \in \mathcal{K}^{<0} \Leftrightarrow \beta = \alpha \sigma_i^{-1},
\]
where $\alpha \in \B_W^-$ and $ l(\alpha)+1 =l(\beta)$.  
\end{lemma}
(In the statement of the above lemma, for $\gamma\in \B_W$,  $l(\gamma)$ denotes the word-length of $\gamma$ in the Artin generators $\sigma_i^{\pm1}$.)

\begin{proof}
The proof is by induction on $l(\beta)$. Suppose that $\beta P_i \in \mathcal{K}^{<0}$.  Let $\beta = \beta' \sigma_j^{-1}$, with $l(\beta')+1=l(\beta)$.  
If $i=j$, we are done, so suppose that $i\neq j$.

If $\sigma_i\sigma_j= \sigma_j\sigma_i$, then $\sigma_j^{-1} P_i \cong P_i$, whence $\beta' P_i \in \mathcal{K}^{< 0}$.  Thus, by, induction, $\beta' = \beta'' \sigma_i^{-1}$, with $l(\beta'') + 1 = l(\beta')$; setting $\alpha = \beta''\sigma_j^{-1}$, it follows that $\beta = \alpha \sigma_i^{-1}$, with $l(\alpha) + 1 = l(\beta)$, as desired.

So now suppose that $\sigma_i \sigma_j \sigma_i = \sigma_j \sigma_i \sigma_j$ and that we have
\[
	\beta' (P_j\langle 1\rangle \rightarrow P_i) \in \mathcal{K}^{<0}.
\]
Since $\beta'$ is a negative braid, $\beta' (P_j)\in \mathcal{K}^{\leq 0}$. From this it follows that $\beta' P_i\in \mathcal{K}^{<0}$; for if 
$\beta' P_i$ has a non-zero $t$-slice in $\calK^{\geq 0}$, then by assumption any terms in $\mathcal{K}^0(\beta' P_i)$ must cancel under Gaussian elimination with terms appearing in $\mathcal{K}^{\geq 1}(\beta' P_j)$.  But since $\beta'$ is a negative braid, there are no terms appearing in $\mathcal{K}^{\geq 1}(\beta' P_j)$.  Thus $\beta' P_i \in \mathcal{K}^{<0}$, and therefore by induction hypothesis, $\beta' = \beta''\sigma_i^{-1}$ with $l(\beta')=l(\beta'')+1$.    Now $\beta'' \sigma_i^{-1}\sigma_j^{-1}P_i\cong \beta''P_j [1]\langle 1\rangle \in \mathcal{K}^{<0}$.  Since the shift $[1]\langle 1 \rangle$ preserves $\mathcal{K}^{<0}$, it follows that $\beta'' P_j\in \mathcal{K}^{<0}$, whence $\beta'' = \beta'''\sigma_j^{-1}$ with $l(\beta'')=l(\beta''')+1$ by the induction hypothesis.

Now
$\beta = \beta''' \sigma_j^{-1}\sigma_i^{-1}\sigma_j^{-1} = (\beta'''\sigma_i^{-1}\sigma_j^{-1})\sigma_i^{-1}$, as desired.

The converse claim is clear, since if $\beta = \beta' \sigma_i^{-1}$, with $\beta'$ a negative braid, then 
$\beta P_i \cong \beta' P_i [1]\langle 2 \rangle$; as $P_i [1]\langle 2 \rangle \in \mathcal{K}^{<0}$, and $\beta'$ is a negative braid, it follows that $ \beta' P_i [1]\langle 2 \rangle\in \mathcal{K}^{<0}$.
\end{proof}

An immediate corollary of the above Lemma that we leave as an exercise to the interested reader is that the sets $X^w$ are pairwise disjoint.

The above lemma also has the following important consequence: 

\begin{corollary} \label{cor:moninj}
The canonical morphisms of monoids $\B_W^+ \longrightarrow \B_W$ and $\B_W^-\longrightarrow \B_W$ are injective.
\end{corollary}
\begin{proof}
By Lemma \ref{lem:neg}, an expression for $\beta\in \B_W^-$ in the monoid generators $\{\sigma_i^{-1}\}$ can be read off inductively from the action of $\beta$ on the indecomposable projective modules $\{P_i\}$.  Thus the morphism of monoids
\[
	\B_W^- \longrightarrow [\mbox{Aut}(\calK)]
\]
is injective, where here $[\mbox{Aut}(\calK)]$ is the group of isomorphism classes of autoequivalences of $\calK$.  
Since the above monoid morphism factors through $\B_W$ via the canonical map $\B_W^-\longrightarrow \B_W$, this canonical map must therefore also be injective.  The proof of injectivity for the map from positive monoid $\B_W^+$ to $\B_W$ follows by a similar argument.
\end{proof}

\begin{lemma}\label{lem:neg_gen}
Let $Y$ be a negative braid complex and let $\beta$ be a negative braid.
Then
\[
\beta Y \in \mathcal{K}^{<0} \iff rf(\beta)(Y)\in \mathcal{K}^{<0}.
\]
\end{lemma}
\begin{proof}
Write $Y = \alpha P_i$ with $\alpha$ a negative braid.  Then
\[
	\beta Y \cong \beta \alpha P_i,
\]
so, by Lemma \ref{lem:neg}, $\beta \alpha$ has $\sigma_i^{-1}$ as a right descent.  Thus $rf(\beta \alpha)$ has $\sigma_i^{-1}$ as a right descent, so one can write $rf(\beta\alpha)=\eta \sigma_i^{-1}$ with $\eta \in \B_W^{-}$. Since $\sigma_i^{-1}P_i\in \calK^{<0}$ and $\eta \calK^{<0}\subset \calK^{< 0}$, then
\[
rf(\beta\alpha) P_i\in \mathcal{K}^{<0}.
\]	
Since
\[
	rf(\beta \alpha) = rf ( rf(\beta) \cdot \alpha),
\]
it follows that $ rf ( rf(\beta)\cdot \alpha) P_i\in \mathcal{K}^{<0}$, and hence that
$rf(\beta) \cdot \alpha P_i \in \mathcal{K}^{<0}$.  

Thus $rf(\beta) Y \in \mathcal{K}^{<0}$.  The converse is clear.

\end{proof}

Now we prove Proposition \ref{prop:Bruhatping}.

\begin{proof}
Suppose $\mathcal{C}\in X^w$.   We show that $\sigma_i\mathcal{C} \in X^{lf(\sigma_i w)}$.
If $Y$ is a negative braid complex,
\[
	\Hom(\sigma_i \mathcal{C} , Y) \cong \Hom(\mathcal{C} ,\sigma_i^{-1} Y).
\]
        Since $\mathcal{C}\in X^{w}$, and $\sigma_i^{-1} Y$ is a negative braid complex, we have
\[
	\mathcal{C}\in X^w \iff w_+^{-1}\sigma_i^{-1} Y\in \mathcal{K}^{<0}.
\]
But now, by the previous Lemma \ref{lem:neg_gen}, 
\[
	w_+^{-1}\sigma_i^{-1} Y\in \mathcal{K}^{<0}\iff rf(w_+^{-1}\sigma_i^{-1})Y\in \mathcal{K}^{<0}.
\]
Since $rf(w_+^{-1}\sigma_i^{-1}) = lf(\sigma_iw_+)^{-1}$, this completes the proof.
\end{proof}

The sets $X^w$ satisfy the requirements to apply Lemma \ref{lem:pingpong}, and thus one obtains a new proof of the main theorem of Brav-Thomas in \cite{BT}:
\begin{corollary} \label{cor:faithfulness}
For $W$ a Weyl group of type ADE, the action of the braid group $\B_W$ on $\calK$ is faithful.
\end{corollary}

%
\subsection{Dual ping pong and the $\vec{o}$-baric structure}
%

In the previous subsection, we used the compatibility between the classical positive and negative monoids and the canonical $t$-structure on $\calK$ to define complexes for standard ping pong.  In this subsection, we use the compatibility of the dual positive and negative monoids with the $\vec{o}$-baric structure on $\calK$ to give analogous dual ping pong 
constructions.

Recall that $\calK_0$ denotes the heart of the $\vec{o}$-baric structure on $\calK$.
For $X\in \calK_0$, we define a positive Bessis braid $\nu^+(X)\in \calBe^+$ and a negative Bessis braid 
$\nu^-(X)\in \calBe^-$ by
\begin{equation}
 \nu_+(X)=\min\{\beta\in \calBe^+|\beta\cdot X\in \calK_{[1,\infty)}\},
\end{equation}
and
\begin{equation}
\nu_-(X)=\min\{\beta\in \calBe^-|\beta\cdot X\in \calK_{(-\infty,-1]}\}.
\end{equation}

Thus $\nu^+(X)$ is the greatest common divisor of all the positive Bessis braids which lift $X$ into the strictly positive part of the $\vec{o}$-baric structure, while $\nu^-(X)$ is the greatest common divisor of all the negative Bessis braids which lower $X$ into the strictly negative part of the $\vec{o}$-baric structure.

\begin{remark}\label{rem:factorizationBe+}
Let $X\in \calK_0$ and $\beta \in \calBe^+$. By the minimality of $\nu_+(X)$, if $\beta\cdot X\in \calK_{[1,\infty)}$ then $\nu_+(X)$ divides $\beta$ in the dual positive monoid.  In fact, as we will show later, $\nu_+(X)(X)\in \calK_{[1,\infty)}$ (see Corollary \ref{cor:nu}).  Thus the set of Bessis braids which lift a fixed complex from the baric heart $\calK_0$ into the positive part $\calK_{\geq 1}$ is closed under greatest common divisors.
\end{remark}
The following propositions are the main results needed in the construction of dual ping pong.

\begin{prop} \label{prop:toolsdualpp}
Let $\tau_u,\tau_t\in \boldsfT\subset \calBe^+$ be associated to reflections $u,t\in T$ with $u\neq t$. Then:
\begin{enumerate}
\item \label{el1} $\tau_t\tau_u\in \calBe^+\Leftrightarrow \tau_t\tau_u\tau_t^{-1}\in \boldsfT \Leftrightarrow \tau_t\mathfrak{C}_u\cong\mathfrak{C}_{tut^{-1}}$;
\item \label{el2} $\tau_t\tau_u\notin \calBe^+\Leftrightarrow \tau_t\tau_u\tau_t^{-1}\notin \boldsfT\Leftrightarrow$
$\tau_t\cdot \calC_u \in \calK_{[0,1]}$ with the following conditions on the baric slices 
\begin{itemize}
\item the top baric slice  $\calK_1(\tau_t\cdot \calC_u)$ is isomorphic to a direct sum of shifts of $\calC_t$;
\item the bottom baric slice $\calK_0(\tau_t\cdot \calC_u)$ satisfies 
$$\nu_+(\calK_0(\tau_t\cdot \calC_u))\tau_t=\lcm(\tau_t,\tau_u)\in \calBe^+.$$
\end{itemize}
\end{enumerate}
\end{prop}

In the above, the statement  $t\mathfrak{C}_u\cong\mathfrak{C}_{tut^{-1}}$ is equivalent to the statement that $\tau_t\calC_u$ is isomorphic to $\calC_{tut^{-1}}\langle k \rangle [-k]$ for some $k\in \Z$. Informally, the second item in the above proposition says that $\tau_t\cdot \calC_u$ looks like
\begin{equation*}
 \tau_t\cdot \calC_u =
\begin{tikzpicture}[anchorbase,scale=.5]
\draw (0,0) rectangle (2,1);
\node at (1,.5) {$X$};
\draw [->] (1.5,1) -- (1.5,1.7);
\node at (2,2.2) {$\bigoplus \calC_t\{-1\}$};
\node at (1.5,-1.4) {};
\end{tikzpicture}
\end{equation*}
where $X=\calK_0(\tau_t\cdot \calC_u)$.

Proposition \ref{prop:toolsdualpp} has a direct translation for the reflections in $\calBe^-$. 
\begin{prop} \label{prop:toolsdualpp_neg}
Let $\tau_u^{-1},\tau_t^{-1}\in\boldsfT^{-1}$ with $u\neq t\in T$. Then:
\begin{enumerate}
\item \label{el1_neg} $\tau_t^{-1}\tau_u^{-1}\in \calBe^-\Leftrightarrow \tau_t^{-1}\tau_u^{-1}\tau_t\in \boldsfT^{-1} \Leftrightarrow \tau_t^{-1}\mathfrak{C}_u=\mathfrak{C}_{t^{-1}ut}$;
\item \label{el2_neg} $\tau_t^{-1}\tau_u^{-1}\notin \calBe^-\Leftrightarrow \tau_t^{-1}\tau_u^{-1}\tau_t\notin \boldsfT^{-1}\Leftrightarrow$
$\tau_t^{-1}\cdot \calC_u \in \calK_{[-1,0]}$ with the following conditions on the baric slices 
\begin{itemize}
\item the bottom baric slice  $\calK_{-1}(\tau_t^{-1}\cdot \calC_u)$ is isomorphic to a direct sum of shifts of $\calC_t$;
\item the top baric slice $\calK_0(\tau_t^{-1}\cdot \calC_u)$ satisfies 
$$\nu_-(\calK_0(\tau_t^{-1}\cdot \calC_u))\tau_t^{-1}=\lcm(\tau_t^{-1},\tau_u^{-1})\in \calBe^-.$$
\end{itemize}
\end{enumerate}
\end{prop}

\begin{prop} \label{prop:toolsdualpp2}
Let $\tau_u$, $\tau_v$ be positive reflections in $\calBe^+$. Then:
\[
\tau_u\tau_v\in \calBe^+\Leftrightarrow \Hom(\calC_u,\mathfrak{C}_v)=0.
\]
\end{prop}

The proof of Propositions \ref{prop:toolsdualpp} and \ref{prop:toolsdualpp_neg} will be postponed until Section~\ref{sec:prooftoolsdualpp}, while Proposition \ref{prop:toolsdualpp2} will be established in Section \ref{sec:prooftoolsdualpp2}.

\begin{example}
We give an example to illustrate the content of Proposition \ref{prop:toolsdualpp2}.  Consider the following orientation of the $A_3$ Dynkin diagram, whose corresponding Coxeter element is $c=s_1s_3s_2$.
\[
\begin{tikzpicture}
\node at (0,0) {$\bullet$};
\node at (1,0) {$\bullet$};
\node at (2,0) {$\bullet$};
\node at (0,.25) {$\scriptstyle 1$};
\node at (1,.25) {$\scriptstyle 2$};
\node at (2,.25) {$\scriptstyle 3$};
\draw [->] (.8,0) -- (.2,0);
\draw [->] (1.2,0) -- (1.8,0);
\end{tikzpicture}
\]

Now, 
\[
	\Hom(P_1,P_2\{k\})\neq 0 \implies k=-1, and 
\] 
\[
	\Hom(P_2,P_1\{k\})\neq 0 \implies k=0.  
\]
From this we see, via Proposition \ref{prop:toolsdualpp2}, that $\sigma_1\sigma_2\in \calBe^+$, but $\sigma_2\sigma_1\notin \calBe^+$. Similarly, since 
\[
	\Hom(P_3,P_1\{k\}) = \Hom(P_3,P_1\{k\}) = 0\text{ for all } k,
\]
we have that $\sigma_1\sigma_3 = \sigma_3\sigma_1 \in \calBe^+$.

On the other hand, consider $\tau_u = \sigma_2$ and $\tau_v=\sigma_1\sigma_3\sigma_2\sigma_3^{-1}\sigma_1^{-1}$; the associated reflection complexes are: 
\[
P_2\;\text{and}\;
\begin{tikzpicture}[anchorbase]
\node at (0,0) {$P_2$};
\node at (1,.3) {$P_1$};
\node at (1,-.3) {$P_3$};
\node at (1,0) {$\oplus$};
\draw [->] (.2,.1) -- (.8,.3);
\draw [->] (.2,-.05) -- (.8,-.25);
\end{tikzpicture}
\]
Then there are non-trivial orientation-degree 0 maps $f$ and $g$ in both directions between these two complexes: 
\[
\begin{tikzpicture}[anchorbase]
\node at (0,0) {$P_2$};
\node at (1,.3) {$P_1$};
\node at (1,-.3) {$P_3$};
\node at (1,0) {$\oplus$};
\draw [->] (.2,.1) -- (.8,.3);
\draw [->] (.2,-.05) -- (.8,-.25);
\node at (0,-1.5) {$P_2$};
\draw [->] (0,-.3) -- (0,-1.2);
\node at (-.2,-.7) {$f$};
\end{tikzpicture}
\quad\text{and}\quad
\begin{tikzpicture}[anchorbase]
\node at (0,0) {$P_2$};
\node at (1,.3) {$P_1$};
\node at (1,-.3) {$P_3$};
\node at (1,0) {$\oplus$};
\draw [->] (.2,.1) -- (.8,.3);
\draw [->] (.2,-.05) -- (.8,-.25);
\node at (1.2,1.5) {$P_2$};
\draw [->] (1.6,1.3) .. controls (1.9,.75) and (1.8,.25) .. (1.2,-.2);
\node at (1.9,.6) {$g$};
\end{tikzpicture}
\]
and indeed, neither $\tau_u\tau_v$ nor $\tau_v\tau_u$ is in $\calBe^+$.
\end{example}

%
\subsection{Dual ping pong}
%

For $w\in \calBe^+$, we define $X_w\subset \calK_{\geq 0}$ as follows:

\begin{equation} \label{Xw}
X_w:=\{C \in \calK_{\geq 0}\mid \Hom(C,\mathfrak{C}_t)=0\; \Leftrightarrow\; t|w\}
\end{equation}
We claim that the sets $X_w$ satisfy the conditions of the dual ping pong lemma \ref{lem:dualpingpong}.

It is clear from the definition that the sets $\{X_w\}_{w\in \calBe^+}$ are disjoint, as $w\in \calBe^+$ is determined by its reflection factors. Now, using Remark \ref{rem:increfl}, it suffices to prove that for $\tau_u$ a Bessis reflection and $w\in \calBe^{+}$, we have $\tau_uX_w\subset X_{lf(\tau_uw)}$. 

To see this, write $\tau_uw=lf(\tau_uw)w'$, and let $t$ be a reflection.

\begin{itemize}
\item If $\tau_u^{-1}\tau_t\tau_u\in \calBe^+$, then using  Proposition \ref{prop:toolsdualpp_neg}, item \ref{el1_neg}, 
we have that 
\[
	\tau_u^{-1}\mathfrak{C}_t\cong\mathfrak{C}_{u^{-1}tu}. 
\]
Then, for $C\in X_w$,
\begin{align*}
\Hom (\tau_uC,\mathfrak{C}_t) \cong 0 & \Longleftrightarrow \Hom(C,\tau_u^{-1}\mathfrak{C}_t)\cong0 
\Longleftrightarrow \Hom(C,\mathfrak{C}_{u^{-1}tu})\cong 0 \\
&\Longleftrightarrow \tau_u^{-1}\tau_t\tau_u|w \Longleftrightarrow \tau_t|\tau_uw \Longleftrightarrow \tau_t|lf(\tau_uw).
\end{align*}
\item If $\tau_u^{-1}\tau_t\tau_u\notin \calBe^+$, then by Proposition \ref{prop:toolsdualpp_neg}, we have that $\tau_u^{-1}\calC_t\in \calK_{[-1,0]}$ with bottom baric slice in $\calK_{-1}(\tau_u^{-1}\calC_t)$ isomorphic to a direct sum of shifts of $\calC_u$. Let us denote the top baric slice $\calK_{0}(\tau_u^{-1}\calC_t)$ by $X_0$; the complex $X_0$ satisfies $\nu_+(X_0)^{-1}\tau_u^{-1}=\lcm(\tau_t^{-1},\tau_u^{-1})$. In a picture, we have:
\[
\tau_u^{-1}\cdot \calC_t=
\begin{tikzpicture}[anchorbase,scale=.5]
\draw (0,0) rectangle (2,1);
\node at (1,.5) {$X_0$};
\draw [->] (.5,-.9) -- (.5,0);
\node at (.5,-1.3) {$\oplus\calC_u\{1\}$};
\node at (1,2.2) {};
\end{tikzpicture}
\]
Since $C\in \calK^{\geq 0}$, $\calC_u\{1\}\in \calK_{<0}$, and there are no $A_\Gamma$-module maps from minimal complexes in $\calK^{\geq 0}$ to minimal complexes in $\calK_{<0}$, we see that
$\Hom(C,\tau_u^{-1}\calC_t)\cong \Hom(C,X_0)$. 

Now, by Theorem \ref{thm:indecCplxes}, $X_0$ is isomorphic to a direct sum of reflection complexes,
\[
	X_0=\oplus_{i\in J}\calC_{t_i}\langle k_i \rangle [k_i], 
\]
so that
\begin{align*}
\Hom(C,\tau_u^{-1}\mathfrak{C}_t)=0 &\Longleftrightarrow \Hom(C,\mathfrak{C}_{t_i})=0\; \forall i\in J \Longleftrightarrow \tau_{t_i}|w \; \forall i \in J.
\end{align*}
Note that the last item is equivalent to $lcm_{i\in J}\{\tau_{t_i}\} | w$.
Now, by Corollary \ref{cor:nu} (which will be proven along with Proposition \ref{prop:toolsdualpp} in the next subsection) 
\[
	\nu_+(X_0)=\lcm(\tau_{t_i}),
\]
and we arrive at
\[
	\Hom(C,\tau_u^{-1}\mathfrak{C}_t)=0 \Longleftrightarrow \nu_+(X_0) | w.
\]
Finally, we note that $\lcm(\tau_t,\tau_u)=\tau_u\nu_+(X_0)$, so that 
\[
	 \nu_+(X_0) | w \Longleftrightarrow lcm(\tau_t,\tau_u) | \tau_uw \Longleftrightarrow \tau_t|\tau_uw \Longleftrightarrow \tau_t|lf(\tau_uw).
\]
\end{itemize}

The last condition we need to check is that the sets $X_w$ are non-empty, and we do this by constructing an explicit complex in each $X_w$.

To do this, write $\gamma=w'w$, with $l_{refl}(w) + l_{refl}(w') = l_{refl}(\gamma)$, and let $C=\bigoplus_{\tau_t|w'}\calC_t$. 
We claim that $C\in X_w$.  Indeed, if $\tau_{t_0}|w$, then $\forall \tau_t|w'$, we have that $\tau_t\tau_{t_0}\in \calBe^+$ and by Lemma \ref{prop:toolsdualpp2}, it follows that $\Hom(C,\mathfrak{C}_{t_0})=0$.  On the other hand, if $\tau_{t_0}|w'$, then $\calC_{t_0}$ is one of the summands $C$ and thus $\Hom(C,\mathfrak{C}_{t_0})\neq 0$.  Finally, if $\tau_{t_0}$ divides neither $w$ nor $w'$, then Lemma \ref{lemma:lcm_comp} implies that there exists a reflection $t_1$ with $\tau_{t_1}| w'$ and $\tau_{t_1}\tau_{t_0}\notin \calBe^+$. By Lemma \ref{prop:toolsdualpp2}, we have $\Hom(\calC_{t_1},\mathfrak{C}_{t_0})\neq 0$.  Thus we conclude that
\[
	\Hom(C,\mathfrak{C}_t) = 0 \iff \tau_t|w,
\]
so that $C\in X_w$.  This completes the proof that the sets $\{X_w\}_{w\in \calBe^+}$ satisfy the requirements of Lemma \ref{lem:dualpingpong} for dual ping pong. As a consequence, we obtain another proof of Corollary~\ref{cor:faithfulness}, that is, of the main theorem of Brav-Thomas in \cite{BT}.

%
\subsection{Proof of Proposition \ref{prop:toolsdualpp}}
\label{sec:prooftoolsdualpp}
%

We begin with the following lemma.

\begin{lemma} \label{lemma:technical1a}
Let $\beta\in \calBe^+$ and let $\calC_u\in \calK_0$ be a reflection complex. We have $\beta\cdot \calC_u \in \calK_{[0,1]}$.
\end{lemma}

\begin{proof}
Recall that any $\tau_t\in \boldsfT$ is of the form $\tau_t=\gamma^k\sigma_{i_1}\cdots \sigma_{i_j}\sigma_{i_{j+1}}\sigma_{i_j}^{-1}\cdots\sigma_{i_1}^{-1}\gamma^{-k}$, with $\gamma=\sigma_{i_1}\cdots \sigma_{i_n}$. Using Lemma \ref{lemma:cShift}, we have that $\gamma^{-k}\calC_u\in \calK_{-k}$. Now, we successively apply letters $\sigma_{i}^{-1}$ in the order they appear in the reduced expression for $\gamma$.  When $\sigma_i^{-1}$ appears to the left of  $\sigma_l^{-1}$ in the expression for $\gamma$, the definition of the orientation grading on the zigzag algebra $A_\Gamma$ implies that $\Hom(P_i,P_l\{m\})=0$ for $m\neq0$.  From this it follows that $\sigma_{i_j}^{-1}\cdots \sigma_{i_1}^{-1} \gamma^{-k}\calC_u\in \calK_{[-k-1,-k]}$. Now. applying $\sigma_{i_{j+1}}$ will yield a complex in $\calK_{[-k-1,-k+1]}$, with 
$\calK_{k+1}(\sigma_{i_{j+1}}\sigma_{i_j}^{-1}\cdots \sigma_{i_1}^{-1} \gamma^{-k}\calC_u)$ isomorphic to
 a direct sum of shifts of $P_{i_{j+1}}$ (or equal to zero).  Moreover, the fact that 
 \[
 	\Hom(P_{i_m},P_{i_{j+1}}\{r\})\neq 0 \implies r=1 \text{ for } m\leq j
\]
implies that the bottom baric slice does not change when we apply $\sigma_{i_{j+1}}$:
\[
\calK_{k-1}(\sigma_{i_{j+1}}\sigma_{i_j}^{-1}\cdots \sigma_{i_1}^{-1} \gamma^{-k}\calC_u) \cong
\calK_{k-1}(\sigma_{i_j}^{-1}\cdots \sigma_{i_1}^{-1} \gamma^{-k}\calC_u).
\]
 
Now, note that if $Y$ is any minimal complex in $\calK_{-k-1}$ whose underlying chain group only contains projective modules of the form $P_{i_1}\{-k-1\},\dots,P_{i_j}\{-k-1\}$, then $\sigma_{i_1}\cdots \sigma_{i_j}(Y) \in \calK_{-k}$.  Since the bottom slice of  
$\sigma_{i_{j+1}}\sigma_{i_j}^{-1}\cdots \sigma_{i_1}^{-1} \gamma^{-k}\calC_u$ is of this form, it follows that 

$$\sigma_{i_1}\cdots \sigma_{i_j}\sigma_{i_{j+1}}\sigma_{i_j}^{-1}\cdots \sigma_{i_1}^{-1}\gamma^{-k}\calC_u\in \calK_{[-k,-k+1]}.$$ 

Thus $\gamma^k\sigma_{i_1}\cdots \sigma_{i_j}\sigma_{i_{j+1}}\sigma_{i_j}^{-1}\cdots \sigma_{i_1}^{-1}\gamma^{-k}\calC_u\in \calK_{[0,1]}$, and we have that $\beta \calC_u\in \calK_{\geq 0}$.

Thus we see that for any reflection $t$, and any $\calC \in \calK_0$, $\tau_t(\calC) \subset \calK_{\geq0}$.  From this it follows that for any $\beta\in \calBe^+$ and any $\calC\in \calK_0$, $\beta\calC \in \calK_{\geq0}$.  Moreover, for any $\beta\in \calBe^+$, there exists $\beta'\in \calBe^+$ with $\gamma=\beta'\beta$.  Since $\gamma(\calC)\in\calK_1\subset \calK_{\leq 1}$ for all $\calC\in \calK_0$, we must therefore have $\beta \calC \in \calK_{\leq 1}$.  This shows that $\beta \calC\in \calK_{[0,1]}$.

\end{proof}

\begin{remark}\label{rem:3}
It also follows from the proof of Lemma \ref{lemma:technical1a} that if $\beta, \beta', \beta'\beta\in \calBe^+$, and $X\in \calK_0$, but $\calK_1(\beta X) \neq 0$, then $\calK_1(\beta'\beta X)\neq 0.$
\end{remark}

The characterization from Lemma \ref{lemma:technical1a} actually gives a criterion for determining whether or not a product of reflections is a Bessis braid:

\begin{lemma}\label{lemma:2}
Let $\tau_{t_1},\dots, \tau_{t_k}\in \boldsfT$. Then 
\[
	\tau_{t_1}\dots \tau_{t_k} \in \calBe^+ \Longleftrightarrow (\tau_{t_1}\cdots \tau_{t_k})Y \in \calK_{[0,1]}
	\text{ for all }Y\in \calK_0.
\]
\end{lemma}

\begin{proof}
It follows immediately from Lemma \ref{lemma:technical1a} that 
\[
\tau_{t_1}\dots \tau_{t_k} | \gamma \implies (\tau_{t_1}\cdots \tau_{t_k})\calK_0 \subset \calK_{[0,1]}.
\]

For the converse, we induct on $k$.  If $k=1$, the claim is clear. For $k>1$, by induction hypothesis and Remark \ref{rem:3}, both $\tau_{t_2}\cdots \tau_{t_k}$ and $(\tau_{t_2}^{-1}\tau_{t_1}\tau_{t_2})\cdots \tau_{t_k}$ are divisors of $\gamma$, so
$\alpha = \lcm(\tau_{t_2}\cdots \tau_{t_k}, (\tau_{t_2}^{-1}\tau_{t_1}\tau_{t_2})\cdots \tau_{t_k})$
divides $\gamma$ as well. The braid $\tau_{t_1}\tau_{t_2}\cdots \tau_{t_k}$ has length $k$ in the reflection generators, and is a multiple of both the length $k-1$ braid $\tau_{t_2}\cdots \tau_{t_k}$ and the length $k-1$ braid $(\tau_{t_2}^{-1}\tau_{t_1}\tau_{t_2})\cdots \tau_{t_k}$.  Since $t_1\neq t_2$, it follows that these two length $k-1$ braids are distinct, and thus that
$\alpha = \tau_{t_1}\tau_{t_2}\cdots \tau_{t_k}$.  Since $\alpha$ divides $\gamma$, this proves the claim.
\end{proof}

\begin{lemma} \label{lemma:technical1b}
Let $\beta\in \calBe^+$ be a Bessis braid and let $\calC_u$ be a reflection complex.  We have the following:
\begin{enumerate}
\item \label{item:tl1} $\beta\cdot \calC_u\in \calK_0\Longleftrightarrow \beta \tau_u\beta^{-1}\in \boldsfT$ and $\tau_u$ does not divide $\beta$;
\item \label{item:tl2} $\beta\cdot \calC_u\in \calK_1\Longleftrightarrow \tau_u|\beta$.
\end{enumerate}
\end{lemma}
\begin{proof}
For (\ref{item:tl1}), let $\tau_u=\mu \sigma_j \mu^{-1}$ as in \ref{eq:dualpositivelift}, so that $\calC_u\cong \mu P_j$.
Suppose that $\beta\calC_u=\beta\mu P_j\in \calK_0$.  Since $\beta \calC_u\in \calK_0$, it follows that $\tau_u$ does not divide $\beta$, for if $\beta = \alpha \tau_u$ with $\alpha\in \calBe^+$, then 
\[
	\beta \calC_u \cong \alpha \calC_u \{-1\} \in \calK_{\geq 1}.
\] 
We will use Lemma \ref{lemma:2} to show that $\beta \tau_u \beta^{-1}$ is in $\boldsfT$.   To see this, let $Y \in \calK_0$, so that $\tau_uY\in \calK_{[0,1]}$.  The top slice
$\calK_1(\tau_uY)$ is either 0 or is isomorphic to a direct sum of shifts of $\calC_u$.
Now, since $\beta \calC_u\in \calK_0$ and $\beta\in \calBe^+$ it follows that 
$\beta \tau_u \cdot Y \subset \calK_{[0,1]}$, too, and by Lemma \ref{lemma:2} it follows that $\beta \tau_u\in \calBe^+$.  This implies that $\beta \tau_u \beta^{-1}\in \boldsfT$.  

For the converse, suppose $\beta \tau_u \beta^{-1}\in \boldsfT$ and $\tau_u$ does not divide $\beta$.  Then 
$\beta \tau_u \in \calBe^+$, whence $\beta \tau_u \calC_u\in \calK_[0,1]$.  Since $\tau_u \calC_u\cong \calC_u\{-1\}\in \calK_1$, this shows that $\beta \calC_u$ must be in $\calK_0$.

We now prove (\ref{item:tl2}). Suppose first that $\beta\calC_u\in \calK_{1}$, so that $\gamma^{-1}\beta\calC_u\in \calK_0$. Since $\gamma^{-1}\beta\in\calBe^{-}$, we can use the analog of item \ref{item:tl1} for dual negative braids to conclude that $\gamma^{-1}\beta \tau_u^{-1}\in \calBe^-$.  Thus $\beta \tau_u^{-1}$ is in $\calBe^+$, and $\tau_u$ is a factor of $\beta$. For the converse, write $\beta = \beta'\tau_u$ with $\beta'\in \calBe^+$ and $\tau_u$ not dividing $\beta'$.  Then, by item (\ref{item:tl1}), $\beta'\calC_u\in \calK_0$.  Thus
$$\beta\calC_u = \beta' \tau_u\calC_u = \beta'\calC_u\{-1\}\in \calK_1,$$
as desired.
\end{proof}

As a consequence, we get the following corollary, which establishes the link between the lattice structure on the interval $[1,\gamma]=\calBe^+$ in the dual positive monoid ${\B_c^\vee}^+$ and the action of $\B_W$ on $\calK$, equipped with its $\vec{o}$-baric structure.

\begin{corollary}\label{cor:nu}
For $X\in \calK_0$, recall that $\nu_+(X) = \gcd\{\beta \in \calBe^+ : \beta X\in \calK_{>0}\}$.
We have
\begin{itemize}
\item $\nu_+(X)X\in \calK_1$, and 
\item $\nu_+(\oplus_{t\in J}\calC_t)=\lcm\{\tau_t\}_{t\in J}$.
\end{itemize}
\end{corollary}
\begin{proof}
Let $X\in \calK_0$.  After possibly applying homological shifts $[k]$ to the indecomposable summands of $X$, we may assume that $X\in \calK^0_0(c)$.  By Theorem \ref{thm:indecCplxes}, we may then write $X=\oplus_{j\in J}\calC_{t_j}$. Now, consider the set 
$\{\beta\in \calBe^+|\beta\cdot X \in \calK_{1}\}$.   By Lemma \ref{lemma:technical1b} (\ref{item:tl2}) above, if $\beta\in\{\beta\in \calBe^+|\beta\cdot X \in \calK_{1}\}$ and $j\in J$, then $\tau_{t_j}$ divides $\beta$.  Thus $\lcm\{\tau_{t_j}\}_{j\in J}$ divides $\beta$, and therefore $\lcm\{\tau_{t_j}\}_{j\in J}$ divides $\nu_+(X)$.  This shows that $\nu_+(X)X\in \calK_1$.  Since $\lcm\{\tau_{t_j}\}_{j\in J} \in \{\beta\in \calBe^+|\beta\cdot X \in \calK_{1}\}$, it now follows that $\nu_+(X) = \lcm\{\tau_{t_j}\}_{j\in J}$.
\end{proof}

As a consequence of Corollary \ref{cor:nu}, together with the relationship between $\gcd$s and $\lcm$s of subsets of $\calBe^+$, we obtain the following relationship between $\nu_+(X)$ and $\nu_-(X)$ for complexes $X\in \calK_0$.

\begin{lemma}\label{lemma:5a}
$\nu_{-}(X)=\nu_+(X)^{-1}$
\end{lemma}

The following Lemma will now complete the proof of Proposition \ref{prop:toolsdualpp}.

\begin{lemma} \label{lemma:technical2}
Let $\tau_t,\tau_u\in \boldsfT$.  Let $X = \calK_0(t\calC_u)$.   Then $\nu_+(X)\tau_t=\lcm(\tau_t,\tau_u)\in \calBe^+$.
\end{lemma}

\begin{proof}
Let us first prove that $\nu_+(X)\tau_t\in \calBe^+$. As in the proof of Corollary \ref{cor:nu}, we may assume that $X=\oplus_{s\in J}\calC_s$. From $\tau_t^{-1}(\tau_t\calC_u)=\calC_u$, we deduce that $\tau_t^{-1}\cdot X\in \calK_{0}$.
Now, by Lemma \ref{lemma:technical1b}, this implies that $\tau_t^{-1}\tau_s^{-1}\in \calBe^{-}$ for all $s\in J$. Using Lemma \ref{lemma:technical1b} again, we see that $\tau_s\calC_t\in \calK_{0}$, and thus $\gamma^{-1}\tau_s\calC_t\in \calK_{-1}$. By Corollary \ref{cor:nu}),  
\[
	\gcd\{\gamma^{-1}\tau_s\}_{s\in J} \calC_t \cong \gamma^{-1}\lcm\{\tau_s\}_{s\in J} \calC_t \in \calK_{<0},
\]
and thus $\lcm(\{\tau_s\}_{J})\cdot \calC_t\in \calK_{0}$. This implies that $\lcm(\{\tau_s\}_{s\in J})t\cdot \calK_0 \subset \calK_{[0,1]}$, and by Lemma \ref{lemma:2}, we obtain $\nu_+(X)\tau_t\in \calBe^+$.

Clearly $\nu_+(X)\tau_t \calC_u\in \calK_{>0}$, so that $\tau_u|\nu_+(X)t$.  Since $\tau_t$ divides $\nu_+(X) t$ as well, we have that $\lcm(\tau_t,\tau_u)$ divides $\nu_+(X)\tau_t$. Now, writing $\lcm(\tau_t,\tau_u)=\zeta \tau_u$, with $\zeta\in \calBe^+$, we have that $\zeta X\in \calK_{>0}$, so that $\nu_+(X)$ divides $\zeta$.  Thus $\nu_+(X)t$ divides $\lcm(\tau_t,\tau_u)$. This shows that $\nu_+(X)t=\lcm(\tau_t,\tau_u)$.
\end{proof}

This completes the proof of Proposition \ref{prop:toolsdualpp}.

%
\subsection{Proof of Proposition \ref{prop:toolsdualpp2}} \label{sec:prooftoolsdualpp2}
%

Up to this point we have described some information about a braid $\beta$ that can be obtained by considering the action of $\beta$ on complexes of $A_{\Gamma}$-modules.  In this section we collect some information directly about the complex of bimodules associated $\beta$.  This will be of use in the proof of Proposition \ref{prop:toolsdualpp2} and later again in Section \ref{sec:comparing}.

Outside of homological degree 0, the chain groups of a minimal complex associated to a braid $\beta$ are direct sums of bigraded $(A_\Gamma,A_\Gamma)$ bimodules of the form $P_i\otimes Q_j \{k\}\langle l\rangle$.  In homological degree 0, the minimal complex is a direct sum of such bimodules, along with a single copy of the bimodule $A_\Gamma\{0\}\langle 0 \rangle$ itself, which by the conventions of this paper sits in homological degree 0.  The following lemma concerns the baric slices of such a complex of bimodules.  In particular, note that the existence of the term $A_\Gamma\{0\}\langle 0 \rangle$ in homological degree 0 of the minimal complex of $B$ implies that $\calK_0(B) \neq 0$.

\begin{lemma} \label{lemma:dualbimod1}
Let $B$ be a complex of $(A_\Gamma, A_\Gamma)$ bimodules associated to a braid.  Suppose that 
for some $k\geq 0$, $\calK_{-k}(B) \neq 0$, but $\calK_{(-\infty,-k-1]}(B) = 0$.  Moreover, in the case $k=0$, so that $\calK_{<0}(B) = 0$, assume that $B$ is not isomorphic to $A_\Gamma$.  Then $\exists i\in I$ such that $\calK_{-k-1}(B\cdot P_i)\neq 0$.  
\end{lemma}
Note that in the statement of the above Lemma, $\calK_{-k}(B)$ is a complex of $(A_\Gamma,A_\Gamma)$ bimodules, while $\calK_{-k-1}(B\cdot P_i)$ is a complex of left $A_\Gamma$ modules.

\begin{proof}
  
 Consider $B^{bot}:=\calK_{-k}(B)$ the degree $-k$ slice of $B$.   We claim that the minimal complex of $B^{bot}$ has a bimodule $P_j\otimes Q_i\{k\}$ as a {\it subcomplex}. This is clear if $k>0$. If $k=0$, then by assumption the chain groups of the minimal complex of $\calK_{0}(B)$ have at least one summand other than the summand $A_\Gamma$. Moreover, since all nonzero maps from $A_\Gamma$ to  $P_j\otimes Q_i$ are of strictly positive degree, it follows that no other summand of the chain group of $\calK_{0}(B)$ is mapped into from $A_\Gamma$ via the differential.  It then follows that the minimal complex of $\calK_0(B)$ has a shift of $P_j\otimes Q_i$ as a subcomplex.

The fact that $P_j\otimes Q_i$ is a subcomplex implies that for some $s\in \Z$,
 \[
  \Hom(B^{bot},P_j\otimes Q_i[s]\{k\})\cong \C \neq 0,
 \]
as the above hom space contains the identity map out of the $P_j\otimes Q_i[s]$ summand of $B$, and when $P_j\otimes Q_i[s]$ is a summand of a minimal complex, this map cannot be homotopic to zero.

Now by Lemma \ref{lem:dualotherside}, we have:
\begin{align*}
  \Hom(B^{bot}\otimes_{A_\Gamma} P_i,P_j[s]\{k+1\}) &\simeq \Hom(B^{bot}\otimes_{A_\Gamma} P_i,P_j[s]\{k\})\{1\} \\
  &\simeq \Hom(B^{bot},P_j\otimes_\C Q_i[s]\{k\})\neq 0.
\end{align*}
The non-vanishing of $ \Hom(B^{bot}\otimes_{A_\Gamma} P_i,P_j[s]\{k+1\})$ implies that $\calK_{-k-1}(B\otimes P_i)\neq 0$, which concludes the proof.
\end{proof}

An analogous argument proves a similar statement for top baric slices:

\begin{lemma} \label{lemma:dualbimod2}
Let $B$ be a minimal complex of bimodules associated to a braid. Assume that for some $k>0$, $\calK_{k}(B)\neq 0$ but for all $l>k$, $\calK_l(B)=0$. Then $\exists i\in I$ such that $\calK_{k}(B\otimes P_i)\neq 0$.
\end{lemma}

As a consequence, we have the following relationship between bounds on the baric slices of the complex of bimodules $B$ associated to a braid $\beta$ and the baric slices of complex of left modules $B\otimes_{A_\Gamma} (\oplus_i  P_i)$.

\begin{corollary} \label{cor:64}
  Let $B$ be a complex of bimodules associated to a braid $\beta$. Then the complex of left modules $B\otimes_{A_\Gamma} P_i\in \calK_{[0,1]}$ for all $i$ if and only if the complex of bimodules $B$ has only non-zero baric slices in degree $0$ and $1$, with 
  \[
\calK_0(B) \cong A_\Gamma
  \]
\end{corollary}
Since we have already shown that a braid $\beta\in \calBe^+$ if and only if $\beta (P_i)\in \calK_{[0,1]}$ for all $i$, it follows that $\beta\in \calBe^+$ if and only if the complex of bimodules $B$ associated to $\beta$ has only non-zero baric slices in degree $0$ and $1$, with 
  \[
\calK_0(B) \cong A_\Gamma
  \]

We may now complete the proof of Proposition \ref{prop:toolsdualpp2}.

\begin{proof}[Proof of Proposition \ref{prop:toolsdualpp2}]
  Recall the statement of Proposition \ref{prop:toolsdualpp2}: if $u$, $v\in T$, then
  \[
\tau_u\tau_v\in \calBe^+\quad \Leftrightarrow \quad \Hom(\calC_u,\mathfrak{C}_v)=0.
  \]

 By Corollary \ref{cor:64}, in order to show that $\tau_u\tau_v\in \calBe^+$, we must show that the complex of bimodules assigned to $\tau_u\tau_v$ has non-trivial baric slices only in degrees 0 and 1, with the degree $0$ baric slice $\calK_0(\tau_u\tau_v) \cong A_\Gamma$.
  
  Let us consider $\tau_u=\mu \sigma_i \mu^{-1}$, so that $\calC_u=\mu P_i\in \calK_0$, and $\tau_v=\delta \sigma_j \delta^{-1}$ with $\calC_v=\delta P_j\in \calK_0$.  The complex $B$ of bimodules associated to $\tau_u\tau_v$ can then be described as follows:
\[
B\cong
\begin{tikzpicture}[anchorbase]
\node at (0,0) {$A_\Gamma$};
\node at (2.5,1) {$\mu P_i \otimes Q_i \mu^{-1}\{-1\}$};
\node at (2.5,-1) {$\delta P_j \otimes Q_j \delta^{-1}\{-1\}$};
\node at (6.5,0) {$\mu P_i \otimes Q_i \mu^{-1}\delta P_j\otimes Q_j \delta^{-1}\{-2\}$};
\draw [->] (.5,.2) -- (1.5,.7);
\draw [->] (.5,-.2) -- (1.5,-.7);
\draw [->] (3.5,.7) -- (4.5,.2);
\draw [->] (3.5,-.7) -- (4.5,-.2);
\end{tikzpicture}
\]

From the fact that $\mu P_i$ and $\delta P_j$ are in $\calK_0$ (and similarly the fact that $Q_i\mu^{-1}$ and $Q_j\delta^{-1}$ are in the baric heart for right $A_\Gamma$ modules), we see that:
\begin{itemize}
\item $\calK_n(B)=0 \text{ for all } n\notin\{0,1,2\};$
\item $\calK_0(B)\cong A_\Gamma;$
\item $\calK_2(B) \cong \calK_2\big (\mu P_i \otimes Q_i \mu^{-1}\delta P_j\otimes Q_j \delta^{-1}\{-2\}\big).$
\end{itemize}

Furthermore, since $\mu P_i$ and $Q_j\delta^{-1}$ are complexes of left (resp. right) $A_\Gamma$ modules in the baric heart, it follows that $\calK_2(B)= 0$ if and only if the degree 0 component of the graded vector space $Q_i\mu^{-1}\delta P_j$ vanishes; that is, if and only if 
$\Hom_{grVect}(\C,Q_i\mu^{-1}\delta P_j[k])=\{0\}$ for all $k\in \Z$.

Then, by adjunction, we see that:
\[
\Hom_{grVect}(\C,Q_i\mu^{-1}\delta P_j[k])\simeq \Hom_{A_\Gamma}(P_i,\mu^{-1}\delta P_j[k])\simeq \Hom_{A_\Gamma}(\mu P_i,\delta P_j[k]).
\]
for all $k\in \Z$.
Thus, we conclude that $uv\in \calBe^+$ if and only if
\[
\Hom(\calC_u,\mathfrak{C}_v)=\{0\},
\]
as desired.
\end{proof}

%
\section{Homological interpretations of word-length metrics}\label{sec:metric}
%

%
\subsection{The dual metric and the $o$-baric structure} \label{subsec:absolutemetric}
%

%
\subsubsection{Word-length}
%

We now would like to give a better description of how the orientation grading we have considered on the category of $A_\Gamma$-modules interacts with the structure of the braid group. The final result, Theorem \ref{thm:1}, will allow one to read off the word-length of a braid in terms of generators from Bessis' braids.

For $\beta$ a braid in the braid group $\B_W$, define:
\begin{align}
\calL^\gamma_+(\beta)&:=\max \{k\in \Z:\forall X\in \calK_0, \calK_k(\beta X)\neq 0\} \\
\calL^\gamma_-(\beta)&:=\max \{k\in \Z:\forall X\in \calK_0, \calK_{-k}(\beta X)\neq 0\} \\
\calL^\gamma(\beta)&:=\max(\calL^\gamma_+(\beta),0)+\max(\calL^\gamma_-(\beta),0)
\end{align}

\begin{theorem} \label{thm:1}
Let $\beta\in \B_W$. Then we have:
\begin{enumerate}
\item \label{item:1thm1} $d(\beta_1,\beta_2)=l(\beta_2^{-1}\beta_1)$ is a metric on $\B_W$;
\item \label{item:2thm1} $\calL_+^\gamma(\beta)$ equals the number of letters from $\calBe^+$ used in any minimal length expression for $\beta$ in the alphabet $\calBe^+\cup \calBe^-$;
\item \label{item:3thm1} $\calL_-^\gamma(\beta)$ equals the number of letters from $\calBe^-$ used in any minimal length expression for $\beta$ in the alphabet $\calBe^+\cup \calBe^-$;
\item \label{item:4thm1} $-\calL^\gamma_+(\beta)$ equals the number of times $\gamma^{-1}$ appears in any minimal length expression for $\beta$ in the alphabet $\calBe^+\cup \calBe^-$;
\item \label{item:5thm1} $-\calL_-^\gamma(\beta)$ equals the number of times $\gamma$ appears in any minimal length expression for $\beta$ in the alphabet $\calBe^+\cup \calBe^-$.
\end{enumerate}
In any of the above statement, if the considered number is negative, it stands for zero letters appearing in the expression. Combined together, these statements tell that $\calL^\gamma$ is the word-length metric $L_c$ in the generators $\calBe^+\cup \calBe^-$.
\end{theorem}

The proof will be based on a succession of technical results that will give a step-by-step description of the image of linear complexes under the action of Bessis braids, and braids that one can form from them. Starting from an expression of a braid $\beta$ as $\beta=\beta^-\beta^+$, with $\beta^-\in \langle \calBe^-\rangle$ and $\beta^+\in \calBe^+$, written as $\beta^+=\beta_1^+\cdots \beta_k^+$ left-greedy (that is, $\beta_1^+$ is the lowest common multiple of the descents of $\beta$, etc.) and $\beta^-=\beta_l^-\cdots \beta_1^-$ right-greedy, the goal will be to control well enough the different layers of the image of a complex under the action of the first terms of $\beta$ so that we know what result applying one more term will yield. In that context, Lemma \ref{lemma:technical2} already gives us some control, that we will want to extend first to any Bessis braid, then to any braid lying in Bessis' positive monoid, and finally to any braid.

%
\subsubsection{Bessis' braids}
%

We first introduce a useful piece of notation.

\begin{definition}\label{def:check}
For $\beta \in \calBe^+$, define:
\[
{}^\vee\beta=\beta^{-1}\gamma\quad,\quad \beta^\vee=\gamma\beta^{-1}
\]
\end{definition}

The following are basic facts about the lattice $\calBe^+$ that we will use freely in what follows.
\begin{itemize}
\item
$
{}^{\vee}(\beta^\vee)=\beta=({}^\vee\beta)^\vee;
$
\item
$
\quad {}^\vee(\alpha\beta)=({}^\vee\beta)\gamma({}^{\vee}\alpha);
$
\item 
$
 \quad ({}^\vee\beta)({}^\vee\alpha)={}^\vee(\alpha \gamma^{-1}\beta); 
$
\item
$\gcd(\beta,\beta^\vee)=1;$
\item
$\lcm(\beta,\beta^\vee)=\gamma$.
\end{itemize}

We can then state the main statement describing the action of a braid $\beta\in \calBe^+$ on a complex in the baric heart $\calK_{0}$.

\begin{lemma} \label{lemma:6}
For $\beta\in \calBe^+$, $\beta(\oplus_{t\in J}\calC_t )\in \calK_{[0,1]}$ with:
\begin{itemize}
\item the bottom baric slice $X_0 = \calK_0(\beta \oplus_{t\in J}\calC_t) \in \calK_0$ satisfies
\[
	\nu_+(X_0)\beta=\lcm(\nu_+(\oplus_{t\in J}\calC_t),\beta);
\]
\item the top baric slice $X_1\{-1\}=\calK_1(\beta(\oplus_{t\in J}\calC_t )\{-1\} \in \calK_1$ satisfies
\[
	\beta=\nu_+(X_1)\gcd((\nu_+(\oplus_{t\in J}\calC_t))^\vee,\beta).
\]
\end{itemize}
\end{lemma}

\begin{proof}
Note that $\nu_+(X_0)\beta (\oplus_{t\in J}\calC_t) \in \calK_{[1,\infty)}$, so we have
\[\nu_+(\oplus_{t\in J}\calC_t)|\nu_+(X_0)\beta.\]
The right hand side is obviously also divisible by $\beta$, so we get: 
\[\lcm(\nu_+(\oplus_{t\in J}\calC_t),\beta)|\nu_+(X_0)\beta.\]

Now, $\lcm(\nu_+(\oplus_{t\in J}\calC_t),\beta) (\oplus_{t\in J}\calC_t)\in \calK_{[1,\infty)}$. We may write
\[
\lcm(\nu_+(\oplus_{t\in J}\mathfrak{C}_t),\beta)=
\delta \beta,
\]
with $\delta \in \calBe^+$ and
\[
	l_{refl}(\lcm(\nu_+(\oplus_{t\in J}\mathfrak{C}_t),\beta) = l_{refl}(\delta)  + l_{refl}(\beta).
\]
From this we see that $\delta(X_0)\in \calK_1$, so $\nu_+(X_0) | \delta$.
Thus we see that
\[
\nu_+(X_0)\beta | \lcm(\nu_+(\oplus_{t\in J}\calC_t),\beta)
\]
and also that
\[
\lcm(\nu_+(\oplus_{t\in J}\calC_t),\beta) | \nu_+(X_0)\beta.
\]
Thus the two braids are equal, proving the first statement.

The second statement follows from the first (or, rather, its analog for braids in $\calBe^-$) once we replace 
$\beta$ by $\gamma^{-1}\beta\in \calBe^-$ . 

\end{proof}

We state some special cases of the above as corollaries for later use.

\begin{corollary} \label{cor:3}
Let $\beta\in \calBe^+$, and let $J$ be such that $lcm\{\tau_t, t\in J\} = \gamma$.  Then
\[
	\nu^+(\calK_0(\beta (\oplus_{t\in J}\calC_t)))=\beta^\vee.
\]

\end{corollary}

\begin{corollary}\label{cor:4}
Let $\beta\in \calBe^+$, and suppose that $\gcd(\nu_+(\oplus_{t\in J}\calC_t)^\vee,\beta)=1$.
Then $\beta=\nu_+(\calK_1(\beta(\oplus_{t\in J}\calC_t))$.
\end{corollary}

%
\subsubsection{The dual positive monoid}
%

The following proposition is the main result describing the top baric slice of a general dual positive braid acting on 
the baric heart $\calK_0$.

\begin{prop} \label{prop:2}
Let $\beta=\beta_k\cdots \beta_1$ with $\beta_i\in \calBe^+$ be right-greedy, and let 
\[
	X_k = \calK_k(\beta\oplus_{t\in T}\calC_t )
\]
Then
\begin{equation*}
\nu_+(X_k)=\beta_k.
\end{equation*} 
\end{prop}
\begin{proof}
When $k=1$, the claim follows from Lemma \ref{lemma:6}.

Let $k>1$. By induction, the claim is true for $\beta_{k-1}\cdots \beta_1$, hence $X_{k-1}$, the degree $k-1$ baric slice  $\beta_{k-1}\cdots \beta_1(\oplus_{t\in T} \calC_t)$, satisfies
\[
	\nu_+(X_{k-1}) = \beta_{k-1}.
\]
Since the decomposition is greedy, $\beta_k\beta_{k-1}\notin \calBe^+$. We want to claim that this implies that 
$\beta_kX_{k-1}\notin \calK_0$.

If $\beta_kX_{k-1}\in \calK_0$, then Lemma \ref{lemma:6} implies that 
\[
	\beta_k=\gcd(\gamma\beta_{k-1}^{-1},\beta_k).
\] 
 In particular, this gives that that $\beta_k|\gamma\beta_{k-1}^{-1}$ and thus that $\beta_k\beta_{k-1}|\gamma$, which contradicts the greediness of the decomposition $\beta_k\beta_{k-1}\dots \beta_1$.
 
 Since $\beta_kX_{k-1}\notin \calK_0$, we have that 
 \[
 	X_k \cong \calK_1(\beta_kX_{k-1})\neq0.
\]
 
 Now, clearly, $\beta_k^{-1}(X_k)\in \calK_{-1}$, whence $\beta_k$ is divisible by $\alpha=\nu_+(X_k)$, and we may write $\beta_k = \alpha \delta$, with $\delta \in \calBe^+$.  We want to show that $\delta = 1$, as that will imply that 
 $\beta_k = \nu_+(X_k)$. 
 
To see that $\delta = 1$, we observe that
\[
	\delta\beta_{k-1}\cdots \beta_1\cdot \calK_0 \subset \calK_{[0,k-1]}.
\]
But, from what we have shown above, this implies that $\delta\beta_{k-1}\cdots \beta_1$ has a greedy decomposition with $k-1$ Bessis braids.  Thus this greedy decomposition must be $\beta_{k-1}\cdots \beta_1$, which shows that $\delta=1$.
\end{proof}

As an immediate corollary, we obtain a theorem of Birman-Ko-Lee~\cite{BKL}, Bessis \cite{Bessis} and Brady-Watt \cite{BW}.

\begin{corollary}\label{cor:dualmonoid}
The map of monoids 
\[
	{\B_c^\vee}^+_W  \longrightarrow \B_W
\]
is injective.
\end{corollary}
\begin{proof}
By Proposition \ref{prop:2}, the map from ${\B_c^\vee}^+_W$ to the group of isomorphism classes of autoequivalences of $\calK$ is injective.  Since this map factors through $ \B_W$, the result follows.
\end{proof}

Proposition \ref{prop:2} described the relationship between the top baric slice and the right-greedy normal form of a dual positive braid.  It turns out that the bottom baric slice is analogously related to the left-greedy normal form.  We will deduce this from Proposition \ref{prop:2} together with Lemma \ref{lemma:10} below, which explains how to pass between left-greedy decompositions of dual positive braids and right-greedy decompositions of dual negative braids.  
We leave the proof of Lemma \ref{lemma:10} as an exercise for the reader.
\begin{lemma} \label{lemma:10}
If $\beta=\beta_1\cdots \beta_k$ with $\beta_i\in \calBe^+$ is left-greedy, then
\begin{equation*}
\gamma^{-k}\beta=(\gamma^{-k}\beta_1\gamma^{k-1})(\gamma^{-k+1}\beta_2\gamma^{k-2})\cdots(\gamma^{-1}\beta_k)
\end{equation*}
is a right-greedy expression of $\gamma^{-k}\beta$ in the generators $\calBe^-$.
\end{lemma}

\begin{corollary} \label{cor:6}
If $\beta=\beta_1\cdots \beta_k$ with $\beta_i\in \calBe^+$ is left-greedy, and let
\[
	X_0 = \calK_0(\beta (\oplus_{t\in T} \calC_t)).
\]
Then $\nu_+(X_0)=\beta_1^\vee$.
\end{corollary}
\begin{proof}
Consider the right-greedy expression for $\gamma^{-k}\beta$ coming from Lemma \ref{lemma:10}. 
Let $X_{-k} = \calK_{-k}(\gamma^{-k}\beta (\oplus_{t\in T}\calC_t)).$
Then by Proposition \ref{prop:2} and Lemma \ref{lemma:5a}, we have that
\[
\nu_+(X_{-k})=\gamma^{-k+1}\beta_1^{-1}\gamma^{k}.
\]

Furthermore, if $X_0 = \calK_0(\beta (\oplus_{t\in T}\calC_t))$, then clearly
\[
X_{-k}\cong\gamma^{-k}X_0.
\]
Thus,
\[
\nu_+(X_0)=\gamma^k\nu_+(X_{-k})\gamma^{-k}=\gamma\beta_1^{-1}=\beta_1^{\vee}.
\]
\end{proof}

%
\subsubsection{General braids} \label{subsec:genbraids_or}
%
We now generalize the results from the previous section from dual positive braids to arbitrary braids.  Recall from Section \ref{subsec:prelimWordLength} that every braid $\beta$ can be written as 
\[
	\beta = \beta^-\cdot \beta^+,
\]
with $\beta^+\in {\B_c^\vee}^+$, $\beta^-\in {\B_c^\vee}^-$.  Furthermore, we may write
$\beta^+=\beta^+_1\cdots\beta^+_k$ left-greedy, and $\beta^-=\beta^-_l\cdots \beta^-_1$ right-greedy, and 
$\gcd(\beta_1^+,\beta_1^-)=1$.

Proposition \ref{prop:2}, together with Corollary \ref{cor:4}, now imply the following result, which gives a linear-algebraic algorithm for computing the greedy normal form of any braid.

\begin{prop}\label{prop:3}
Let $\beta = \beta^-\cdot \beta^+$ be a reduced minimal decomposition, with
$\beta^- = \beta^-_l\cdots \beta^-_1$ right-greedy.  Let $X_{-l} = \calK_{-l}(\beta \oplus_{t\in T} \calC_t)$.  Then we have
$\nu_-(X_{-l})=\beta^-_l$.
\end{prop}

We can now prove Theorem \ref{thm:1}.

\begin{proof}[Proof of Theorem \ref{thm:1}]
  Let us first prove that $d$ is a metric. The only non-trivial thing to check is that $d(\beta_1,\beta_2)=0$ implies $\beta_1=\beta_2$, or in other words that
  \[
  	\beta_2^{-1}\beta_1\cdot \calK_0 \subset \calK_0\implies \beta_2^{-1}\beta_1=1.
\]
Writing the braid $\beta_2^{-1}\beta_1$ in a reduced minimal decomposition, Proposition \ref{prop:3} shows that $l$, the length of the dual negative braid in such a decomposition, is equal to 0. Thus $\beta_2^{-1}\beta_1$ is a dual positive braid, and Proposition \ref{prop:2} then implies that the braid must be trivial.

The remaining items follow from Proposition \ref{prop:3}.  For example, to prove \ref{item:3thm1}, write  $\beta = \beta^-\beta^+$ in a reduced minimal decomposition.  Then, by Proposition \ref{prop:3}, $\calL_-^\gamma(\beta)$ equals the minimal number of generators from $\calBe^{-}$ requested to write $\beta^-$.

Item \ref{item:2thm1} follows similarly.  For item \ref{item:5thm1}, suppose that $\beta (\oplus_{t\in T} \calC_t) \in \calK_{[a,\infty)}$ with $a = -\calL^\gamma_-(\beta) > 0$.  Then both $\gamma^{-a}\beta$ is a dual positive braid.  Moreover, since 
\[
\calK_0(	\gamma^{-a}\beta \oplus_{t\in T} \calC_t) \neq 0,
\]
it follows that $\gamma^{-a}\beta$ does not have
$\gamma$ as one of its letters in a minimal length decomposition.  Thus $\beta$ has exactly $a$ letters equal to $\gamma$ in a minimal length decomposition.  Item \ref{item:4thm1} is proven similarly.
\end{proof}


%
\subsection{The standard metric and the canonical $t$-structure}
%

As in the oriented graded case, for $\beta$ a braid in the braid group $\B_W$, define:
\begin{align}
\calL'_+(\beta)&:=\max \{k\in \Z:\forall X\in \calK^0, \calK^k(\beta X)\neq 0\} \\
\calL'_-(\beta)&:=\max \{k\in \Z:\forall X\in \calK^0, \calK^{-k}(\beta X)\neq 0\} \\
\calL'(\beta)&:=\max(\calL'_+(\beta),0)+\max(\calL'_-(\beta),0)
\end{align}

We now prove an analog of Theorem \ref{thm:1} for the word-length metric in the Weyl generators $L$ and the canonical $t$-structure.  The content of the theorem below is very similar to the main result of Brav-Thomas in \cite{BT}.

\begin{theorem} \label{thm:2}
Let $\beta\in \B_W$. Then we have: 
\begin{enumerate}
\item \label{item:1thm2} $d'(\beta_1,\beta_2)=\calL'(\beta_2^{-1}\beta_1)$ is a metric on $\B_W$;
\item \label{item:2thm2} $\calL'_+(\beta)$ equals the number of letters from $\mathcal{W}^+$ used in any minimal length expression for $\beta$ in the alphabet $\mathcal{W}^+\cup \mathcal{W}^-$;
\item \label{item:3thm2} $\calL'_-(\beta)$ equals the number of letters from $\mathcal{W}^-$ used in any minimal length expression for $\beta$ in the alphabet $\mathcal{W}^+\cup \mathcal{W}^-$;
\item \label{item:4thm2} $-\calL'_+(\beta)$ equals the number of times $\Delta^{-1}$ appears in any minimal length expression for $\beta$ in the alphabet $\mathcal{W}^+\cup \mathcal{W}^-$;
\item \label{item:5thm2} $-\calL'_-(\beta)$ equals the number of times $\Delta$ appears in any minimal length expression for $\beta$ in the alphabet $\mathcal{W}^+\cup \mathcal{W}^-$.
\end{enumerate}
In any of the above statement, if the considered number is negative, it stands for zero letters appearing in the expression. Combined together, these statements tell
that $\calL'$ is the word-length metric in the Weyl generators $\mathcal{W}^+\cup \mathcal{W}^-$.
\end{theorem}

As before, this theorem will be a consequence of some technical results.

%
\subsubsection{Understanding the action}
%

For $i\in I$, set $\mathfrak{P}_i=\oplus_{k\in \Z} P_i[k]\langle k \rangle$.
For $X\in \calK$ and $r\in \Z$, we define the following subsets of the vertex set $I$: 
\begin{align}
E_r^+(X)=\{i : \Hom(\mathfrak{P}_i\langle -r\rangle,X)\neq 0\} \label{eq:E+}, \\
E_r^-(X)=\{i : \Hom(X,\mathfrak{P}_i\langle -r\rangle)\neq 0\} \label{eq:E-}.
\end{align}
For example, let $\Gamma$ be of type $A_2$, so that $I = \{1,2\}$.  Let $X$ be the complex:
\[
\sigma_2^2\sigma_1 P_2 = \quad
\begin{tikzpicture}[anchorzero]
\node at (0,0) {$P_1$};
\node at (1.5,0) {$P_2\langle -1\rangle$};
\node at (1.5,1) {$P_2\langle -3 \rangle$};
\draw [->] (.3,0) -- (.9,0);
\draw [->] (1.5,.3) -- (1.5,.7);
\end{tikzpicture}
\]
Then we have $E_0^-=\{1\}$, $E_0^+=\emptyset$ and $E_1^-=\emptyset$, $E_1^+=\{2\}$. 

We first prove some facts about the $E_r^{\pm}$ sets for positive Weyl braids $\beta\in \mathcal{W}^+$. This essentially recovers part of \cite[Proposition 3.1]{BT}.

\begin{prop} \label{prop:11a}
Let $\beta \in \mathcal{W}^+$, with $\pi(\beta) = w\in W$. Then 
\[
	E_1^+(\beta \oplus_{r\in I} P_r)= D_L(w).
\]
\end{prop}

\begin{proof}

The proof goes by induction on the length $l(w)$.  When $l(w)=1$, the result is immediate, as $E_1^+(\sigma_i \oplus_{r\in I} P_r)= \{i\} = D_L(s_i)$.

Suppose now that the claim is proven for elements in $W$ of length $k$, and let $\beta\in \mathcal{W}^+$ with 
$\pi(\beta) = w$ and $l(w)= k+1$.  

We first prove that  $ E_1^+(\beta (\oplus_{r\in I} P_r)) \subset D_L(w)$.
Suppose that
\[
\Hom(\mathfrak{P}_i\langle -1\rangle,\beta (\oplus P_r))\neq \{0\}.
\]
That $s_i\in D_L(w)$ follows easily by induction hypothesis if $\beta = \sigma_j \mu$ with $\sigma_i\sigma_j = \sigma_j\sigma_i$.  For in that case:
\[
\Hom(\mathfrak{P}_i\langle -1\rangle,\sigma_j\mu (\oplus P_r))=\Hom(\mathfrak{P}_i\langle -1\rangle,\mu(\oplus P_r))\neq\{0\},
\]
and by induction hypothesis, we have that $s_i\in D_L^+(\pi(\mu))$; since $s_is_j=s_js_i$,
we conclude that $s_i\in D_L^+(w)$.

So suppose that $\beta = \sigma_j \mu$ with $l(\pi(\beta)) = l(\pi(\mu))+1$, and $\sigma_i\sigma_j \sigma_i= \sigma_j\sigma_i\sigma_j$.
In that case, we have:
\begin{equation}\label{eq:chainmap}
0\neq \Hom(\mathfrak{P}_i\langle -1 \rangle,\sigma_j\mu (\oplus P_r))=\Hom((\mathfrak{P}_j \rightarrow \mathfrak{P}_i\langle -1\rangle),\mu (\oplus P_r)).
\end{equation}

Denote $\mu(\oplus P_r)=\oplus X_r$ the decomposition of the complex $\mu(\oplus P_r)$ in terms of the homological grading, and consider a non-zero representative $(f,g)$ of $\Hom((P_j\langle -k\rangle \rightarrow P_i\langle -k-1\rangle)[-k],\mu(\oplus P_r))$ as illustrated below:
\[
\begin{tikzpicture}[anchorbase]
  \node (P1) at (0,0) {$P_j\langle -k\rangle$};
  \node (P2) at (2.5,0) {$P_i\langle -k-1\rangle$};
  \draw [->] (P1) -- (P2);
  \node at (1.25,.3) {$p$};
  \node (dotsl) at (-4,-2) {$\cdots$};
  \node (Xk-2) at (-2.5,-2) {$X_{k-2}$};
  \node (Xk-1) at (0,-2) {$X_{k-1}$};
  \node (Xk) at (2.5,-2) {$X_k$};
  \node (Xk+1) at (5,-2) {$X_{k+1}$};
  \node (dotsr) at (6.5,-2) {$\cdots$};
  \draw [->] (dotsl) -- (Xk-2);
  \draw [->] (Xk-2) -- (Xk-1);
  \node at (-1.25,-1.7) {$\partial$};
  \draw [->] (Xk-1) -- (Xk);
  \node at (1.25,-1.7) {$\partial$}; 
  \draw [->] (Xk) -- (Xk+1);
  \draw [->] (Xk+1) -- (dotsr);
  \node at (3.75,-1.7) {$\partial$};
  \draw [->] (P1) -- (Xk-1);
  \node at (.2,-1) {$f$};
  \draw [->] (P2) -- (Xk);
  \node at (2.7,-1) {$g$};
\end{tikzpicture}
\]

Then, $g$ can be restricted to a chain map in $\Hom(\mathfrak{P_i}\langle -1\rangle,\mu(\oplus P_r))$. If it is not homotopic to zero, then this proves that $\Hom(\mathfrak{P_i}\langle -1\rangle, \mu(\oplus P_r))\neq \{0\}$ and by induction that $s_i$ is a descent of $\pi(\mu)$. Thus $\beta=\sigma_j\sigma_i\beta'$ with $\beta'\in \mathcal{W}^+$ and $l(w)=l(\pi(\beta')+2$ and 
\[
	\Hom(\mathfrak{P}_i\langle -1 \rangle,\sigma_j\sigma_i\beta'(\oplus P_r))=\Hom(\mathfrak{P}_j\langle -1 \rangle,\beta'(\oplus P_r)) \neq 0.
\]
This shows that $s_j$ is a left descent of $\pi(\beta')$ and, using the braid relation, that $s_i$ indeed is a left descent of $\pi(\beta)$.

Now, if $g$ were homotopic to zero with homotopy map $h:P_i[-k]\langle-k-1\rangle \rightarrow X_{k-1}[-(k-1)]$, then $f+hp\in \Hom(\mathfrak{P_i}\langle -1\rangle,\mu (\oplus P_r)$ would be a chain map non-homotopic to zero, for if it were, $(f,g)$ would also have been homotopic to zero. But since $s_j$ is not a descent of $\pi(\mu)$, by induction hypothesis we know that 
\[
\Hom(\mathfrak{P}_j\langle -1 \rangle, \mu(\oplus P_r))=0.
\]

We now prove the inclusion $D_L(w)\subset E_1^+(\beta \oplus_{r\in I} P_r)$.
Suppose $\beta=\sigma_i\beta'$, $\pi(\beta) = w$, $\pi(\beta') = w'$, and so that $w=s_iw'$, $k+1=l(w)>l(w')$.  
We want to show that $i\in E_1^+(\beta)$.
Suppose that $\beta'=\sigma_j\mu$, with $l(w') > l(\pi(\mu)$. If $\sigma_i$ and $\sigma_j$ commute, the claim follows by induction hypothesis, as then $\beta = \sigma_i\sigma_j \mu = \sigma_j\sigma_i\mu$, with 
$i \in D_L(s_i\pi(\mu)) = E_1^+(\sigma_i\mu)$.  Now, since $\sigma_j P_i \cong P_i$, it follows that $i\in E_1^+(\beta)$.

Now suppose $\beta'=\sigma_j\mu$ with $s_is_js_i = s_js_is_j$ in $W$. 
Then, since $s_j$ is a descent of $\pi(\beta')$, we have by induction hypothesis that:
\[
\Hom(\mathfrak{P_j}\langle-1\rangle, \beta'(\oplus P_r))\neq \{0\}.
\]
This implies that
\[
\Hom((\mathfrak{P}_j\langle-1\rangle\rightarrow \mathfrak{P}_i\langle -2\rangle), \sigma_i\beta'(\oplus P_r))\neq \{0\}.
\]
Let $(f,g)$ be chain map between minimal complexes representing a non-zero element of 
$\Hom((\mathfrak{P}_j\langle-1\rangle\rightarrow \mathfrak{P}_i\langle -2\rangle), \sigma_i\beta'(\oplus P_r))$, as illustrated below.
\[
\begin{tikzpicture}[anchorbase]
  \node (P1) at (0,0) {$P_j\langle -k\rangle$};
  \node (P2) at (2.5,0) {$P_i\langle -k-1\rangle$};
  \draw [->] (P1) -- (P2);
  \node at (1.25,.3) {$p$};
  \node at (-4.5,-2) {$\sigma_i\beta'(\oplus P_r):$};
   \node (Xk-2) at (-2.5,-2) {$X_{k-2}$};
  \node (Xk-1) at (0,-2) {$X_{k-1}$};
  \node (Xk) at (2.5,-2) {$X_k$};
  \node (Xk+1) at (5,-2) {$X_{k+1}$};
  \draw [->] (-3.5,-2) -- (Xk-2);
  \draw [->] (Xk-2) -- (Xk-1);
  \node at (-1.25,-1.7) {$\partial$};
  \draw [->] (Xk-1) -- (Xk);
  \node at (1.25,-1.7) {$\partial$}; 
  \draw [->] (Xk) -- (Xk+1);
  \node at (3.75,-1.7) {$\partial$};
  \draw [->] (Xk+1) -- (5.8,-2);
  \draw [->] (P1) -- (Xk-1);
  \node at (.2,-1) {$f$};
  \draw [->] (P2) -- (Xk);
  \node at (2.7,-1) {$g$};
\end{tikzpicture}
\]

Then $g$ is a well-defined chain map.  If $g$ is not homotopic to 0, then we are done, as then 
$\Hom(\mathfrak{P}_i \langle -1 \rangle, \beta(\oplus P_r))\neq 0$.  On the other hand, if $g$ is homotopic to 0, then following the same argument as in the first half of the above proof, it follows that
$\Hom(\mathfrak{P}_j \langle -1 \rangle, \beta(\oplus P_r))\neq 0$, in which case $s_j$ is a descent of $w$.  In that case, since $s_i$ and $s_j$ are both descents of $w$, it follows that 
\[
	\beta = \sigma_i\sigma_j\sigma_i \alpha = \sigma_j\sigma_i\sigma_j \alpha
\]
for $\alpha\in \mathcal{W}^+$ with $l(\pi(\alpha))+3 = l(w)$.

By induction hypothesis, we have:
\[
\Hom(\mathfrak{P}_j\langle-1\rangle,\sigma_j\alpha(\oplus P_r))\neq \{0\},
\]
from which we deduce:
\[
\Hom(\sigma_j\sigma_i\mathfrak{P}_j\langle -1\rangle,\beta(\oplus P_r))=\Hom(\mathfrak{P}_i\langle -1\rangle,\beta (\oplus P_r))\neq \{0\}
\]
which concludes the proof.

\end{proof}

The following lemma is the analog for the classical Garside monoid of the dual Garside monoid's Lemma \ref{lemma:technical1a}.

\begin{lemma} \label{lemma:11b}
Let $\beta\in \mathcal{W}^+$. Then $\beta \cdot(\oplus P_r)\in \calK^{[0,1]}$.
\end{lemma}

\begin{proof}
The proof goes by induction. When $l(\pi(\beta))=1$, the claim follows by explicit computation.

Let us now take $\beta=\sigma_k \beta'$, with $l(\pi(\beta)) = l = l(\pi(\beta')) + 1$.  Note that this implies that 
$s_k\notin D_L(\pi(\beta'))$.
By induction hypothesis, $\beta'(\oplus P_r)\in \calK^{[0,1]}$.  Now, if 
\[
	\calK^2 (\sigma_k\beta' (\oplus P_r))\neq 0,
\]
then we must have 
\[
	\Hom(\mathfrak{P}_k\langle -2 \rangle,\beta \oplus P_r) \cong
	\Hom(\mathfrak{P}_k\langle -1 \rangle,\beta'\oplus P_r) \neq 0,
\]
whence we would have $s_k\in D_L(\pi(\beta'))$.
From this it follows that $\beta (\oplus P_r) \in \calK^{[0,1]}$, as desired.
 \end{proof}

The above, together with (the positive analog of) Lemma \ref{lem:neg} now implies the following.
\begin{lemma}\label{lemma:11c}
Let $\beta\in \mathcal{W}^+$. Then $\beta P_i\in \calK^1\Longleftrightarrow s_i\in D_R(\pi(\beta))$.
\end{lemma}

\begin{lemma} \label{lemma:DeltaShift}
$\Delta \mathfrak{P}_i \cong \mathfrak{P}_j \langle -1 \rangle \in \calK^{1}$, with $\Delta \sigma_i\Delta^{-1}=\sigma_j$.
\end{lemma}
\begin{proof}
Lemma \ref{lemma:11b} ensures that $\Delta P_i\in \calK^{[0,1]}$, and Lemma \ref{lemma:11c} shows that since $\sigma_i$ is a descent of $\Delta$, we have $\Delta P_i\in \calK^1$. 

Now, after passing to the Grothendieck group and setting $q=-1$, $\Delta$ acts as the long element of $W$.  In particular, the action of $\Delta$ on the $q=-1$ Grothendieck group permutes the simple roots according to an automorphism of the Dynkin diagram. 
Since the complex $\Delta P_i\in \calK^1$, the minimal complex is completely determined by its image in the Grothendieck group.  Since the image of $\Delta P_i$ in the Grothendieck group is a simple root $\alpha_j$, it follows that $\Delta P_i$ is isomorphic to (a shift of) $P_j$.
\end{proof}

The following Lemma also essentially appears in \cite{BT}.
\begin{lemma} \label{lemma:BT}
Let $\beta=\beta_k\cdots \beta_1\in \B_W^+$ be right-greedy.  Then 
\[
	E_k^{+}(\beta\cdot (\oplus_{\sigma_i\in D_R^+(\beta_1)}P_i)) = D_L(\pi(\beta_k)).
\]
\end{lemma}
Moreover, the conclusion remains the same if we replace  be the same if we replace $\oplus_{i\in D_R^+(\pi(\beta_1))}P_i$ by $\oplus_{r\in I}P_r$.
\begin{proof}
The equality follows by induction once the base case $k=1$ case is proven, and that case is a direct consequence of Lemma \ref{lemma:11c}.
\end{proof}

As a corollary, we obtain an analogous statement relating the greedy decomposition of $\beta\in \B_W^+$ to the action of $\beta$ on $( \oplus_{i}P_i)$.

\begin{corollary} \label{cor:5}
Let $\beta=\beta_1\cdots \beta_k\in \B_W^+$ be left-greedy.  Then
\[ 
 E_0^{-}(\beta\cdot \oplus P_i) = I - D_L(\pi(\beta_1)).
\]
\end{corollary}

\begin{proof}
First, notice that $\Delta^{-k}\beta\in \B_W^-$.  Set 
\[
	\alpha_i = \Delta^{-k+i} (\Delta^{-1} \beta_i) \Delta^{k-i},
\]
and note that
\[
	\Delta^{-k}\beta = \alpha_1 \alpha_2\dots \alpha_k
\]
is right-greedy.
From the negative braid analog of Lemma \ref{lemma:BT}, we see that 
\[
	E_{-k}^-(\Delta^{-k}\beta(\oplus P_i)) = D_L(\alpha_1).
\]
Now, it is also straightforward to see that
\[
E_{-k}^-(\Delta^{-k}\beta (\oplus P_i))=w_0^{-k}(E_{0}^-(\beta  \oplus P_i))w_0^k,
\]
where $w_0$ is the long element of $W$.
Thus
\[
E_{0}^-(\beta  \oplus P_i)=	w_0^{k}D_L(\alpha_1)w_0^{-k}.
\]
Since 
\[
	D_L(\pi(\alpha_1)) = w_0^{k-1}D_L(w_0^{-1}\pi(\beta_1))w_0^{-k+1},
\]
it follows that
\[
	E_{0}^-(\beta  \oplus P_i) = w_0D_L(\pi(\beta_1))w_0^{-1} = I - D_L(\pi(\beta_1)),
\]
as desired.
\end{proof}

%
\subsubsection{General braids} \label{subsec:genbraids_ex}
%

We have a situation very similar to Section \ref{subsec:genbraids_or}. Indeed, recall from Section \ref{subsec:prelimWordLength} that every braid $\beta$ can be written as 
\[
	\beta = \beta^-\cdot \beta^+,
\]
with $\beta^+\in \B_W^+$, $\beta^-\in \B_W^-$.  Furthermore, we may write $\beta^+=\beta^+_1\cdots\beta^+_k$ left-greedy, and $\beta^-=\beta^-_l\cdots \beta^-_1$ right-greedy, and $\gcd(\beta_1^+,\beta_1^-)=1$.

We can now deduce the following statement for arbitrary $\beta\in \B_W$. We also re-emphasize the point that in order to read the descents of a braid $\beta$, it suffices to compute the minimal complexes of $\beta$ on a finite number of objects (namely, the indecomposable projective modules $P_r$; the computation of minimal complexes, in turn, requires a finite number of linear algebraic computations.

\begin{prop}\label{prop:4}
Let $\beta = \beta^-\beta^+$ be a reduced minimal $(-,+)$ decomposition, with 
\[
	\beta^- = \beta_l^- \dots \beta_1^-
\]
right-greedy.
Then $E_l^-(\beta\oplus P_r)=D_L(\pi(\beta_l^-))$.
\end{prop}

\begin{proof}
Corollary \ref{cor:5} shows that $D_R(\beta_1^-)\subset E_0^-(\beta^+\oplus P_r)$. Then the conclusion follows from Lemma \ref{lemma:BT}.
\end{proof}

This completes the proof of Theorem \ref{thm:2}.

%
\section{Standard and dual word-length metrics and a conjecture of Digne-Gobet}
\label{sec:comparing}
%

It appears as a very natural question to ask what kind of mutual control both gradings can have, giving, for example, more information about the length in the usual or the dual generators. We give here a specific result that solves a conjecture of Digne and Gobet \cite[Conjecture 8.7]{DigneGobet}, that can be restated as follows.

\begin{theorem} \label{thm:DGConj}
Let $\beta\in \calBe^+$, then $\beta \in [\Delta^{-1},\Delta]$ (that is, $\beta=\beta^+\beta^-$ with $\beta^+$, $(\beta^-)^{-1}\in \mathcal{W}^+$).
\end{theorem}

\begin{proof}
To prove the above theorem, rather than studying the action of $\beta$ on linear complexes of $A_{\Gamma}$-modules, we will focus on the gradings on the category of complexes of bimodules.  We wish to show that the complex of bimodules associated to $\beta$ is in the heart of the canonical $t$-structure on the homotopy category of $(A_{\Gamma},A_\Gamma)$ bimodules; that is, that in the minimal complex of $\beta$, all terms of the form
$P_iQ_j\{m\}\langle n\rangle$ lie in homological degree $-n$.

The proof has two main steps.
The first point is to note that the assumption $\beta\in \calBe^+$ implies that in the minimal complex of bimodules, all terms of the form $P_iQ_j\{m\}\langle n\rangle$ have $m=-1$.  (This is a direct consequence of Lemmas \ref{lemma:dualbimod1} and \ref{lemma:dualbimod2}.)

Having established this, it now follows that the only maps which contribute to the boundary in the minimal complex of $\beta$ are either maps into or out of the identity bimodule $A_\Gamma$, or are bimodule maps of orientation degree $\{0\}$.  The possible orientation degree $\{0\}$ bimodule maps are then the following:
\begin{itemize}
\item $P_iQ_j \{-1\}\langle n \rangle \rightarrow P_iQ_j\{-1\}\langle n \rangle$;
\item $P_iQ_j\{-1\} \langle n \rangle \rightarrow P_iQ_{j'} \{-1\}\langle n-1 \rangle$, with the edge between $j$ and $j'$ oriented from $j'$ to $j$;
\item $P_iQ_j\{-1\} \langle n \rangle \rightarrow P_{i'}Q_{j} \{-1\}\langle n-1 \rangle$, with the edge between $i$ and $i'$ oriented from $i$ to $i'$;
\item $P_iQ_j\{-1\}\langle n \rangle \rightarrow P_{i'}Q_{j'}\{-1\}\langle n-2 \rangle$, with oriented edges $i$ to $i'$ and $j'$ to $j$.
\end{itemize}

The complex of bimodules for $\beta$ is in the heart of the canonical $t$-structure if and only if the first and last kinds of maps do not occur in the minimal complex.  Since the space of orientation-degree $\{0\}$ maps is such that
\[
	\dim \Hom(P_iQ_j, P_iQ_j) = 1,
\]
all maps of the first kind are multiplies of the identity map $1: P_iQ_j \{-1\}\langle n \rangle \rightarrow P_iQ_j\{-1\}\langle n \rangle$, and hence do not occur in the minimal complex.  Thus we must show that maps of the last kind
\[
	P_iQ_j\{-1\}\langle n \rangle \rightarrow P_i'Q_j'\{-1\}\langle n-2 \rangle
\]
do not appear in the boundary map of the minimal complex of $\beta$.

In order to show that these maps do not appear in the minimal complex, we need a little bit of setup.  Starting from the Dynkin diagram $\Gamma$, we construct a larger graph $\widehat{\Gamma}$ by adding a single new vertex $x$, and connecting this new vertex to every other vertex of $\Gamma$.  The indecomposable projective (right, respectively, left) $A_{\widehat{\Gamma}}$ modules are $Q_i$, respectively $P_i$ for $i\in \Gamma$ and $Q_x$, respectively $P_x$.  Note that, by construction,
\[
	Q_i P_x \cong \C \text{ for all } i \in \Gamma.
\]

The inclusion of graphs $\Gamma \hookrightarrow \widehat{\Gamma}$ induces an inclusion of zigzag algebras 
\[
	A_{\Gamma} \hookrightarrow A_{\widehat{\Gamma}}.
\]
The induction functor
\[
	\mbox{Ind}:A_{\Gamma}-\mbox{mod} \longrightarrow A_{\widehat{\Gamma}}-\mbox{mod} ,
\]
sends the indecomposable projective $A_{\Gamma}$ module $P_i = A_\Gamma e_i$ to the 
indecomposable projective $A_{\widehat{\Gamma}}$ module $A_{\widehat{\Gamma}} e_i= P_i$.  (Note that we have abused notation slightly here by using the same notation for modules over different algebras.)
Under the induction functor, the homotopy category of graded projective $A_{\Gamma}$-modules embeds as a full triangulated subcategory of the homotopy category of graded projective $A_{\widehat{\Gamma}}$ modules; similarly, the homotopy category of graded projective $(A_{\Gamma},A_{\Gamma})$-bimodules embeds as a full triangulated subcategory of the homotopy category of graded projective $(A_{\widehat{\Gamma}},A_{\widehat{\Gamma}})$-bimodules.  Moreover, a complex of graded, projective $A_{\Gamma}$ modules (or bimodules) is a minimal complex if and only if it is minimal when induced and regarded as a complex of graded, projective $A_{\widehat{\Gamma}}$ modules.

Now, in order to show that maps of $(A_\Gamma,A_\Gamma)$ bimodules of the form 
\[
	P_iQ_j\{-1\}\langle n \rangle \rightarrow P_i'Q_j'\{-1\}\langle n-2 \rangle
\]
do not appear in the minimal complex of $\beta$, we will consider the complex for $\beta$ in the homotopy category of $(\widehat{A}_\Gamma,\widehat{A}_\Gamma)$ bimodules, and show that there are no such maps there.  The reason for this consideration is that we can study the complex of bimodules $(\widehat{A}_\Gamma,\widehat{A}_\Gamma)$ by letting it act on the new indecomposable project module $P_x$.

Now we consider the complex of projective $A_{\widehat{\Gamma}}$ modules $\beta P_x$.  
The minimal complex of $\beta P_x$ has a unique summand of the chain group isomorphic to $P_x\{0\}\langle 0 \rangle$, coming from the unique summand $A_{\widehat{\Gamma}}$ in the minimal complex of bimodules for $\beta$ tensored with $P_x$.  We consider what happens to each of the kinds of bimodule maps considered above when we tensor with the left $A_{\widehat{\Gamma}}$-module $P_x$.
\begin{itemize}
\item Consider first bimodule maps of the form $P_iQ_j\{-1\} \langle n \rangle \rightarrow P_iQ_{j'} \{-1\}\langle n-1 \rangle$, with the edge between $j$ and $j'$ oriented from $j'$ to $j$.  When we act on $P_x$, we obtain the 0 map
\[
	P_i\{-1\}\langle n \rangle \xrightarrow{0} P_i \{-1\} \langle n-1 \rangle.
\]
\item Consider bimodule maps of the form $P_iQ_j\{-1\} \langle n \rangle \rightarrow P_{i'}Q_{j} \{-1\}\langle n-1 \rangle$, with the edge between $i$ and $i'$ oriented from $i$ to $i'$.  When we act on $P_x$, we obtain
\[
	P_i\{-1\}\langle n \rangle \xrightarrow{f} P_{i'} \{-1\} \langle n-1 \rangle.
\]
where
\[
	f\in \Hom(P_i,P_{i'}\langle -1 \rangle) \cong \C.
\]
\item Finally, consider maps of the form $P_iQ_j\{-1\}\langle n \rangle \rightarrow P_i'Q_j'\{-1\}\langle n-2 \rangle$, with oriented edges $i$ to $i'$ and $j'$ to $j$.  When we act on $P_x$, we obtain 
\[
	P_i\{-1\}\langle n \rangle \xrightarrow{0} P_i' \{-1\} \langle n-2 \rangle.
\]
Note that this map must be zero, because
\[
	\Hom(P_i,P_{i'}\langle -2 \rangle) \cong 0
\]
when $i\neq i'$.
\end{itemize}
There are two important things to note regarding the above maps. The first is that none of them are isomorphisms, and thus the complex of modules obtained by taking the minimal complex for $\beta$ and acting on $P_x$ term-by-term is already minimal.  The second thing to note is that, since no non-zero maps of the form $P_i \rightarrow P_i\{-1\}\langle -2 \rangle$ are obtained when we act on $P_x$, the complex $\beta P_x$ is isomorphic to a direct sum of its $t$-slices:
\[
	\beta P_x \cong  \oplus_l \calK^l(\beta P_x)
\]
Since $P_x$ itself is an indecomposable complex, and $\beta$ is an equivalence, it follows that $\beta P_x$ is an indecomposable complex, and so all but one of the summands above must be 0; thus $\beta P_x \cong \calK^l (\beta P_x)$ for some $l\in \Z$.  Since the unique term of the form $P_x$ appears in the minimal complex of $\beta P_x$, we must have $l=0$.  Now from this it follows that no map of the form $P_iQ_j\{-1\}\langle n \rangle \rightarrow P_i'Q_j'\{-1\}\langle n-2 \rangle$ can appear in the minimal complex for $\beta$.

This proves that the minimal complex of bimodules associated to $\beta$ is in the heart of the canonical $t$-structure on the homotopy category of bigraded bimodules.
Using Lemma \ref{thm:2}, it follows that $\beta=\beta^+\beta^-$ with $\beta^+,(\beta^-)^{-1}\in \mathcal{W}^+$.
\end{proof}


\newcommand{\etalchar}[1]{$^{#1}$}
\providecommand{\bysame}{\leavevmode\hbox to3em{\hrulefill}\thinspace}
\providecommand{\MR}{\relax\ifhmode\unskip\space\fi MR }
\providecommand{\MRhref}[2]{%
  \href{http://www.ams.org/mathscinet-getitem?mr=#1}{#2}
}
\providecommand{\href}[2]{#2}


\begin{thebibliography}{DDG{\etalchar{+}}15}

\bibitem[AT10]{AcharTreumann}
P.~N. Achar and D.~Treumann, \emph{Baric structures on triangulated categories
  and coherent sheaves}, International Mathematics Research Notices (2010),
  3688--3743, \href{http://arxiv.org/abs/0808.3209}{arXiv:0808.3209}.

\bibitem[BB05]{BjBr}
A.~Bj{\"o}rner and F.~Brenti, \emph{Combinatorics of {C}oxeter groups},
  Graduate Texts in Mathematics, vol. 231, Springer, New York, 2005.

\bibitem[BDM02]{BeDiMi}
D.~Bessis, F.~Digne, and J.~Michel, \emph{Springer theory in braid groups and
  the {B}irman--{K}o--{L}ee monoid}, Pacific journal of mathematics
  \textbf{205} (2002), no.~2, 287--309,
  \href{http://arxiv.org/abs/math/0010254}{arXiv:math/0010254}.

\bibitem[Bes03]{Bessis}
D.~Bessis, \emph{The dual braid monoid}, Annales scientifiques de l’Ecole
  normale sup{\'e}rieure, vol.~36, 2003,
  \href{http://arxiv.org/abs/math/0101158}{arXiv:math/0101158}, pp.~647--683.

\bibitem[Bez06]{Bez_ICM}
R.~Bezrukavnikov, \emph{Noncommutative counterparts of the {S}pringer
  resolution}, International Congress of Mathematicians (Madrid), vol.~II, Eur.
  Math. Soc., 2006,
  \href{https://arxiv.org/abs/math/0604445}{arXiv:math/0604445},
  pp.~1119--1144.

\bibitem[BG]{BG}
B.~Baumeister and T.~Gobet, \emph{Simple dual braids, noncrossing partitions
  and {M}ikado braids of type {$D_n$}}, in preparation.

\bibitem[BGP73]{BGP}
I.~N. Bern{\v s}te{\u\i n}, I.~M. Gel$\prime$fand, and V.~A. Ponomarev,
  \emph{Coxeter functors, and {G}abriel's theorem}, Uspehi Mat. Nauk
  \textbf{28} (1973), no.~2(170), 19--33.

\bibitem[BKL98]{BKL}
J.~Birman, K.~H. Ko, and S.~J. Lee, \emph{A new approach to the word and
  conjugacy problems in the braid groups}, Advances in Mathematics \textbf{139}
  (1998), no.~2, 322--353,
  \href{http://arxiv.org/abs/math/9712211}{arXiv:math/9712211}.

\bibitem[BS72]{BrieskornSaito}
E.~Brieskorn and K.~Saito, \emph{Artin-gruppen und {C}oxeter-gruppen},
  Inventiones mathematicae \textbf{17} (1972), no.~4, 245--271.

\bibitem[BT11]{BT}
C.~Brav and H.~Thomas, \emph{Braid groups and {K}leinian singularities},
  Mathematische Annalen \textbf{351} (2011), no.~4, 1005--1017,
  \href{http://arxiv.org/abs/0910.2521}{arXiv:0910.2521}.

\bibitem[BW08]{BW}
T.~Brady and C.~Watt, \emph{Non-crossing partition lattices in finite real
  reflection groups}, Transactions of the American Mathematical Society
  \textbf{360} (2008), no.~4, 1983--2005,
  \href{http://arxiv.org/abs/math/0501502}{arXiv:math/0501502}.

\bibitem[DDG{\etalchar{+}}15]{DDGKM}
P.~Dehornoy, F.~Digne, E.~Godelle, D.~Kramer, and J.~Michel, \emph{{Foundations
  of Garside theory}}, Z\"urich: European Mathematical Society (EMS), 2015
  (English).

\bibitem[Del72]{Deligne72}
P.~Deligne, \emph{Les immeubles des groupes de tresses
  g{\'e}n{\'e}ralis{\'e}s}, Inventiones mathematicae \textbf{17} (1972), no.~4,
  273--302.

\bibitem[DG15]{DigneGobet}
F.~Digne and T.~Gobet, \emph{Dual braid monoids, {M}ikado braids and positivity
  in {H}ecke algebras}, 2015,
  \href{http://arxiv.org/abs/1508.06817}{arXiv:1508.06817}.

\bibitem[DP99]{DehornoyParis}
P.~Dehornoy and L.~Paris, \emph{Gaussian groups and {G}arside groups, two
  generalisations of {A}rtin groups}, Proceedings of the London Mathematical
  Society \textbf{79} (1999), no.~3, 569--604.

\bibitem[Gab72]{Gabriel}
P.~Gabriel, \emph{Unzerlegbare darstellungen i}, Manuscripta mathematica
  \textbf{6} (1972), no.~1, 71--103.

\bibitem[GM13]{GelfandManin}
S.~I. Gelfand and Y.~I. Manin, \emph{Methods of homological algebra}, Springer
  Science \& Business Media, 2013.

\bibitem[GTW15]{GTW}
A.~Gadbled, A.-L. Thiel, and E.~Wagner, \emph{Categorical action of the
  extended braid group of affine type {A}}, Communications in Contemporary
  Mathematics (2015), 1650024,
  \href{http://arxiv.org/abs/1504.07596}{arXiv:1504.07596}.

\bibitem[HK01]{huerfano2001category}
R.~S. Huerfano and M.~Khovanov, \emph{A category for the adjoint
  representation}, Journal of Algebra \textbf{246} (2001), no.~2, 514--542.

\bibitem[IUU10]{ishii2010stability}
A.~Ishii, K.~Ueda, and H.~Uehara, \emph{Stability conditions on
  an-singularities}, Journal of Differential Geometry \textbf{84} (2010),
  no.~1, 87--126, \href{http://arxiv.org/abs/math/0609551}{arXiv:math/0609551}.

\bibitem[Jen15]{Jensen}
L.~T. Jensen, \emph{The 2-braid group and {G}arside normal form}, 2015,
  \href{http://arxiv.org/abs/1505.05353}{arXiv:1505.05353}.

\bibitem[KS02]{KhS}
M.~Khovanov and P.~Seidel, \emph{{Quivers, Floer cohomology, and braid group
  actions}}, \href{http://dx.doi.org/10.1090/S0894-0347-01-00374-5}{J. Am.
  Math. Soc.} \textbf{15} (2002), no.~1, 203--271.

\bibitem[{Lic}16]{Licata_free}
A.~M. {Licata}, \emph{{On the 2-linearity of the free group}}, 2016,
  \href{http://arxiv.org/abs/1606.06444}{arXiv:1606.06444}.

\bibitem[MOS09]{MOS}
V.~Mazorchuk, S.~Ovsienko, and C.~Stroppel, \emph{Quadratic duals, koszul dual
  functors, and applications}, Transactions of the American Mathematical
  Society \textbf{361} (2009), no.~3, 1129--1172,
  \href{http://arxiv.org/abs/math/0603475}{arXiv:math/0603475}.

\bibitem[Ric08]{Riche}
S.~Riche, \emph{Geometric braid group action on derived categories of coherent
  sheaves}, Represent. Theory \textbf{12} (2008), 131--169, With a joint
  appendix with Roman Bezrukavnikov.

\bibitem[RZ03]{RouquierZimmermann}
R.~Rouquier and A.~Zimmermann, \emph{Picard groups for derived module
  categories}, Proceedings of the London Mathematical Society \textbf{87}
  (2003), no.~01, 197--225.

\end{thebibliography}
\end{document}